\documentclass[12pt,english]{amsart}
\usepackage[T1]{fontenc}
\usepackage[utf8]{inputenc}
\usepackage{amstext}
\usepackage{amsthm}
\usepackage{amssymb}

\makeatletter
\numberwithin{equation}{section}
\numberwithin{figure}{section}
\theoremstyle{plain}

\theoremstyle{definition}

\theoremstyle{plain}

\theoremstyle{plain}

\usepackage[english]{babel}
\usepackage[all]{xy}
 \usepackage{amsthm}
\usepackage{mathrsfs}
\usepackage{enumerate}
\usepackage{tikz}
\usetikzlibrary{shapes.geometric}
\usepackage{physics}
\usepackage[notcite, final, notref]{showkeys}
\usepackage{dsfont}
\newtheorem{theorem}{Theorem}[section]
\newtheorem{proposition}[theorem]{Proposition}\newtheorem{corollary}[theorem]{Corollary}\newtheorem{definition}[theorem]{Definition}\theoremstyle{definition}
\newtheorem{remark}[theorem]{Remark}

\newcommand{\w}{\omega}

\usepackage{tikz}
\usetikzlibrary{arrows,shapes,positioning,shadows,trees,matrix}
\usepackage{xcolor}

\title[]{Brownian motion: the hyperbolic number setting}

\author[D. Alpay]{Daniel Alpay}
\address{(DA) Schmid College of Science and Technology \\
Chapman University\\
One University Drive
Orange, California 92866\\
USA}
\email{alpay@chapman.edu}

\author[I. Cho]{Ilwoo Cho}
\address{(IC) Saint Ambrose University \\
Department of Mathematics and Statistics\\
518 W. Locust St.
Davenport, Iowa 52803\\
USA}
\email{choilwoo@sau.edu}

\author[L. Mayats-Alpay]{Liora Mayats-Alpay}
\address{(LMA) Computational and Data Sciences, Schmid College of Science and Technology\\
Chapman University\\
One University Drive
Orange, California 92866\\
USA}
\email{mayatsalpay@chapman.edu}

\thanks{D. Alpay thanks the Foster G. and Mary McGaw Professorship in
  Mathematical Sciences, which supported his research.}

\makeatother

\usepackage{babel}
\providecommand{\corollaryname}{Corollary}
\providecommand{\definitionname}{Definition}
\providecommand{\propositionname}{Proposition}
\providecommand{\theoremname}{Theorem}

\begin{document}
\maketitle
\begin{abstract}
The purpose of this paper is to define normal Gaussian variables in
the setting of hyperbolic probabilities, and introduce an associated
Brownian motion, when both the index and the values of the process
lie in the real algebra $\mathbb{H}$ of hyperbolic numbers. 
 In Hida's white noise
space, we construct two probability measures (say $P_1$ and $P_2$), and associate to them two families of $N(0,1)$ variables 
$(Z_n)_{n\in\mathbb N_0}$ (independent with respect to $P_1$) and $(W_n)_{n\in\mathbb N_0}$ (independent with respect to $P_2$). An important feature is that the $Z_n$ and $W_m$ need not be mutually independent either with respect to $P_1$ or $P_2$. An hyperbolic
normal Gaussian variable is constructed (in non-degenerate cases)
from two classical Gaussian variables and the hyperbolic Brownian motion
is, in general, composed from two copies of the classical Brownian
motion. Using the associated Gelfand triples we also compute the derivative
of the hyperbolic Brownian motion as a stochastic distribution.
The argument extends to the $\mathbb{H}$-valued fractional
Brownian motion, and more generally to a wide family of $\mathbb{H}$-valued
stationary-increment second order processes.
\end{abstract}

\maketitle
\tableofcontents{}

\section{Introduction}
\label{sec-int}
\setcounter{equation}{0} The Brownian motion plays a key role in
mathematics, and it is of interest to consider its extensions to different
settings. For instance the paper \cite{MR4710618} considers Brownian
motion indexed by the subsets of a sigma-algebra, and the paper
\cite{abreu2025analyticcontinuationtimebrownian} defines and studies
a Brownian motion indexed by the complex numbers. Let now $\mathbb{H}$
denote the hyperbolic numbers (see Section \ref{sec-2}). In the present
work we construct an $\mathbb{H}$-valued Brownian motion and its
derivative, with index in $\mathbb{H}$. To set the framework, we
first recall that the classical Brownian motion can be built as follows.
Let $Z_{0},Z_{1},\ldots$ be a countable family of independent normal
Gaussian variables in the probability space $(\Omega,\mathcal{A},P)$.
For such a construction, see e.g. \cite[pp. 38-39]{Neveu68}, where $Z_{0},Z_{1},\ldots$
are chosen to be the coordinate functions of the probability space $\left(\mathbb{R},\frac{1}{\sqrt{2\pi}}e^{-\frac{x^{2}}{2}}dx\right)^{\mathbb{N}_{0}}$ endowed with the
product probability. Furthermore, let $\zeta_{0},\zeta_{1},\ldots$ denote the normalized physicists Hermite functions, that is 
\begin{equation}
\zeta_{n}(z)=\frac{(-1)^{n}}{\pi^{1/4}2^{n/2}(n!)^{1/2}}e^{\frac{z^{2}}{2}}\left(\frac{\partial}{\partial z}\right)^{n}e^{-z^{2}},\quad n=0,1,2,\ldots\label{hnnorm}
\end{equation}
See e.g. \cite[(1.1.18) p. 4]{thangavelu1993lectures}.
The formula 
\begin{equation}
B_{t}(\w)=\sum_{n=0}^{\infty}\left(\int_{0}^{t}\zeta_{n}(x)dx\right)Z_{n}(\w),
\end{equation}
where the convergence is in the underlying $L^2(\Omega,\mathcal A,P)$, is a realization of the Brownian motion and its covariance function
is 
\begin{equation}
E_{P}(B_{t}B_{s})=\frac{|t|+|s|-|t-s|}{2},\quad t,s\in\mathbb{R},\label{classical-bm}
\end{equation}
where $E_{P}$ denotes the mathematical expectation with respect to
$P$.\smallskip{}

As is well known, for almost all $\w$ the function $t\mapsto B_{t}(\w)$
has no derivative a.e. On the other hand, using Hida's white noise
space theory, and the notion of Gelfand triple and stochastic distributions
one can interpret the sum 
\[
N_{t}(\w)=\sum_{n=0}^{\infty}\zeta_{n}(t)Z_{n}(\w)
\]
as a stochastic distribution and give a precise meaning to the formula
\[
\frac{dB_{t}}{dt}=N_{t},
\]
where $\frac{dB_{t}}{dt}$ is a limit computed in an underlying topological vector space. See \cite{new_sde}.\smallskip{}

The purpose of this paper is to define and study a version of the
Brownian motion, where both the index set and the values of the process
belong to the algebra of hyperbolic numbers, here denoted by $\mathbb{H}$.
In terms of matrix representations, $\mathbb{H}$ consists of the
matrices of the form 
\begin{equation}
p=\begin{pmatrix}a & b\\
b & a
\end{pmatrix}=U\begin{pmatrix}a+b & 0\\
0 & a-b
\end{pmatrix}U,\quad{\rm with}\quad U=\frac{1}{\sqrt{2}}\begin{pmatrix}1 & 1\\
1 & -1
\end{pmatrix},\quad a,b\in\mathbb{R}.
\end{equation}
See \eqref{eq2} below. \smallskip{}

In the process of replacing the real numbers by the hyperbolic numbers,
we need to define, or recall the definitions of, positive definite
functions and Gaussian variables in the hyperbolic setting. $\mathbb{H}$-valued
probabilities were introduced and studied in \cite{MR3651492}. The
fact that $\mathbb{H}$ is a lattice plays a key role in the arguments;
more precisely, given $p$ and $q$ in $\mathbb{H}$, with respective
matrix representations 
\begin{equation}
p=U\begin{pmatrix}\lambda_{1} & 0\\
0 & \lambda_{2}
\end{pmatrix}U\quad{\rm and}\quad q=U\begin{pmatrix}\mu_{1} & 0\\
0 & \mu_{2}
\end{pmatrix}U\label{p}
\end{equation}
they have uniquely defined greatest lower bound and least upper bound, namely 
\begin{equation}
p\wedge q=U\begin{pmatrix}\lambda_{1}\wedge\mu_{1} & 0\\
0 & \lambda_{2}\wedge\mu_{2}
\end{pmatrix}U\quad{\rm and}\quad p\vee q=U\begin{pmatrix}\lambda_{1}\vee\mu_{1} & 0\\
0 & \lambda_{2}\vee\mu_{2}
\end{pmatrix}U
\label{min-sup}
\end{equation}
(for the definition of the order, see \eqref{order-}).
Besides constructing the hyperbolic Brownian motion we also consider its derivative,
as a $\mathbb{H}$-valued stochastic distribution, and study the case of stationary increment processes in the hyperbolic setting.
In the classical case, these are Gaussian processes with covariance functions of the form 
\begin{equation}
K(t,s)=\int_{\mathbb{R}}\frac{e^{-itu}-1}{u}\frac{e^{isu}-1}{u}d\sigma(u)\label{k-t-s}
\end{equation}
where $d\sigma$ is a positive measure subject to $\int_{\mathbb{R}}\frac{d\sigma(u)}{u^{2}+1}<\infty$.
We will focus on the real-valued case. This happens if and only if
$d\sigma(u)=-d\sigma(-u)$ and then \eqref{k-t-s} takes the form
\begin{equation}
K_{r}(t,s)=\int_{\mathbb{R}}\frac{(1-\cos tu)(1-\cos su)+(\sin tu)(\sin su)}{u^{2}}d\sigma(u).\label{k-t-s-1}
\end{equation}
We will also consider the derivatives of such processes. As in
the papers \cite{aal2,aal3} we look at the case where $d\sigma(u)=m(u)du$,
where the measurable positive function $m(u)$ satisfies 
\begin{equation}
  \int_{\mathbb{R}}\frac{m(u)du}{u^{2}+1}<\infty,
  \label{m-u-2}
\end{equation}
and adapt the approach of these papers to the real-valued case (i.e
when $m(u)$ is even). The case
\[
m(u)=\frac{|u|^{1-2H}}{2\pi},\quad H\in(0,1),
\]
corresponds to the fractional Brownian motion with Hurst parameter $H$, with $H=1/2$ giving back the Brownian motion. More precisely,

\begin{equation}
  \begin{split}
    K_{r}(t,s)&=\frac{1}{2\pi}\int_{\mathbb{R}}\frac{(1-\cos tu)(1-\cos su)+(\sin tu)(\sin su)}{u^{2}}|u|^{1-2H}du\\
&=\frac{V_H}{\pi}\left(|t|^{2H}+|s|^{2H}-|t-s|^{2H}\right)
    \label{k-t-s-1-2}
\end{split}
\end{equation}
where
\begin{equation}
  \label{V-H}
  V_H=\int_0^\infty\frac{1-\cos u}{u^2}u^{1-2H}du=\begin{cases}
\frac{\Gamma(2-2H)\cos(H\pi)}{(1-2H)2H},\quad H\in(0,1)\setminus\left\{ \frac{1}{2}\right\} \\
\,\,\frac{\pi}{2},\,\hspace{2.4cm}H=\frac{1}{2}.
\end{cases}
\end{equation}
The complex-valued case involves extension to the bi-complex numbers (rather than the hyperbolic numbers), and will be considered elsewhere.\smallskip{}

\section{Hyperbolic numbers, matrices, and Probability}

\setcounter{equation}{0} \label{sec-2} 

\subsection{Hyperbolic Numbers}

This section is essentially of a review nature. We set the notation
and refer to \cite{MR3309382,MR3410909,MR1094818} for more information.
Hyperbolic numbers are defined as expressions of the form $p=a+bk$,
where $a,b\in\mathbb{R}$ and $k\not\in\mathbb{R}$ commutes with the
real numbers and satisfies $k^{2}=1$. Taking 
\begin{equation}
k=\begin{pmatrix}0 & 1\\
1 & 0
\end{pmatrix}
\end{equation}
we obtain a matrix representation of hyperbolic numbers, and write
\begin{equation}
p=\begin{pmatrix}a & b\\
b & a
\end{pmatrix}=U\begin{pmatrix}a+b & 0\\
0 & a-b
\end{pmatrix}U,\quad{\rm with}\quad U=\frac{1}{\sqrt{2}}\begin{pmatrix}1 & 1\\
1 & -1
\end{pmatrix}.\label{eq2}
\end{equation}
We set 
\[
e_{+}=\frac{1+k}{2}\quad{\rm and}\quad e_{-}=\frac{1-k}{2}.
\]
Note that 
\begin{eqnarray}
e_{+}^{2} & = & e_{+},\\
e_{-}^{2} & = & e_{-},\\
e_{+}e_{-} & = & 0,\\
e_{+}+e_{-} & = & 1,\\
e_{+}-e_{-} & = & k.
\end{eqnarray}
Thus 
\[
\begin{split}p & =a+kb\\
 & =a(e_{+}+e_{-})+(e_{+}-e_{-})b\\
 & =(a+b)e_{+}+(a-b)e_{-},
\end{split}
\]
the latter being called the idempotent decomposition of the hyperbolic
number $p$. Thus, as a $\mathbb{R}$-vector space, 
\[
\mathbb{H}=\mathrm{span}_{\mathbb{R}}\left\{ 1,k\right\} \overset{\mathrm{iso}}{=}\mathrm{span}_{\mathbb{R}}\left\{ e_{+},e_{-}\right\} ,
\]
satisfying 
\[
\mathbb{R}e_{+}\cap\mathbb{R}e_{-}=\left\{ 0=0+0k\right\} ,\;\;\mathrm{in\;\;}\mathbb{H},
\]
i.e., 
\[
\mathbb{H}=\left(\mathbb{R}e_{+}\right)\oplus_{\mathbb{R}}\left(\mathbb{R}e_{-}\right),
\]
where $\oplus_{\mathbb{R}}$ denotes the direct product of $\mathbb{R}$-algebras.\smallskip

We endow $\mathbb H$ with the involution
\[
\left(x+yk\right)^{\circledast}=x-yk,\;\;\;\;\forall x+yk\in\mathbb{H}.
\]
Note that 
\[
  e_{-}=e_{+}^{\circledast}.
  \]
The algebra $\mathbb{H}$ is equipped with a complete $\mathbb{R}$-norm
$\left\Vert .\right\Vert $, defined by 
\[
\left\Vert x+yk\right\Vert \overset{\mathrm{def}}{=}\sqrt{x^{2}+y^{2}},\;\;\;\;\forall x+yk\in\mathbb{H},
\]
see \cite{alpay-cho-l2,MR4920356}. So, we understand the hyperbolic numbers
$\mathbb{H}$ as a Banach $*$-algebra over $\mathbb{R}$.\smallskip

Let now $f$ be a real-valued function defined on the real numbers. One can
extend (in a non-unique way) the function $f$ to be from $\mathbb{H}$ into itself via the formula 
\begin{equation}
f_{\mathbb{H}}(p)=U\begin{pmatrix}f(\lambda_{1}) & 0\\
0 & f(\lambda_{2})
\end{pmatrix}U\,.
\end{equation}

For instance, with $f(x)=|x|$ we have 
\begin{equation}
  \label{p-q}
|p-q|_{\mathbb{H}}=U\begin{pmatrix}|\lambda_{1}-\mu_{1}| & 0\\
0 & |\lambda_{2}-\mu_{2}|
\end{pmatrix}U\,.
\end{equation}
Given two elements $p=\lambda_{1}e_{+}+\lambda_{2}e_{-}$ and $q=\mu_{1}e_{+}+\mu_{2}e_{-}$
we say that $p\ge q$ if 
\begin{equation}
  \label{order-}
\lambda_{j}\ge\mu_{j},\quad j=1,2.
\end{equation}
Equivalently $p\ge q$ if and only if the same order holds for the
corresponding matrix representations. In particular, with $q=0$ we
have $p\ge0$ if and only if $\lambda_{1}\ge0$ and $\lambda_{2}\ge0$.
In terms of the representation $p=a+bk$, we have 
\begin{equation}
p\ge0\quad\iff\quad\begin{pmatrix}a & b\\
b & a
\end{pmatrix}\ge0\quad\iff\quad\begin{cases}
a+b\ge0,\quad{\rm and}\\
a-b\ge0.
\end{cases}
\label{2-10}
\end{equation}

\begin{definition} The closed unit interval of $\mathbb{H}$ is 
\begin{equation}
\mathbb{I}=\left\{ p\in\mathbb{H}\,\,;\,0\le p\le1\right\} .
\end{equation}
See figure \ref{fig-1}.
\end{definition}

In terms of the representations of $p$ it is easily seen that:

\begin{proposition} Let $p=a+bk=\lambda_{1}e_{+}+\lambda_{2}e_{-}$.
Then the following are equivalent:\\
 $(1)$ $p\in\mathbb{I}$.\\
 $(2)$ $\lambda_{j}\in[0,1]$, $j=1,2$.\\
 $(3)$ $a$ and $b$ satisfy the inequalities 
\[
\begin{split}-a & \le b\le1-a,\\
a-1 & \le b\le a.
\end{split}
\]
\end{proposition}

  \begin{figure}[h]
        \includegraphics[width=0.8\linewidth]{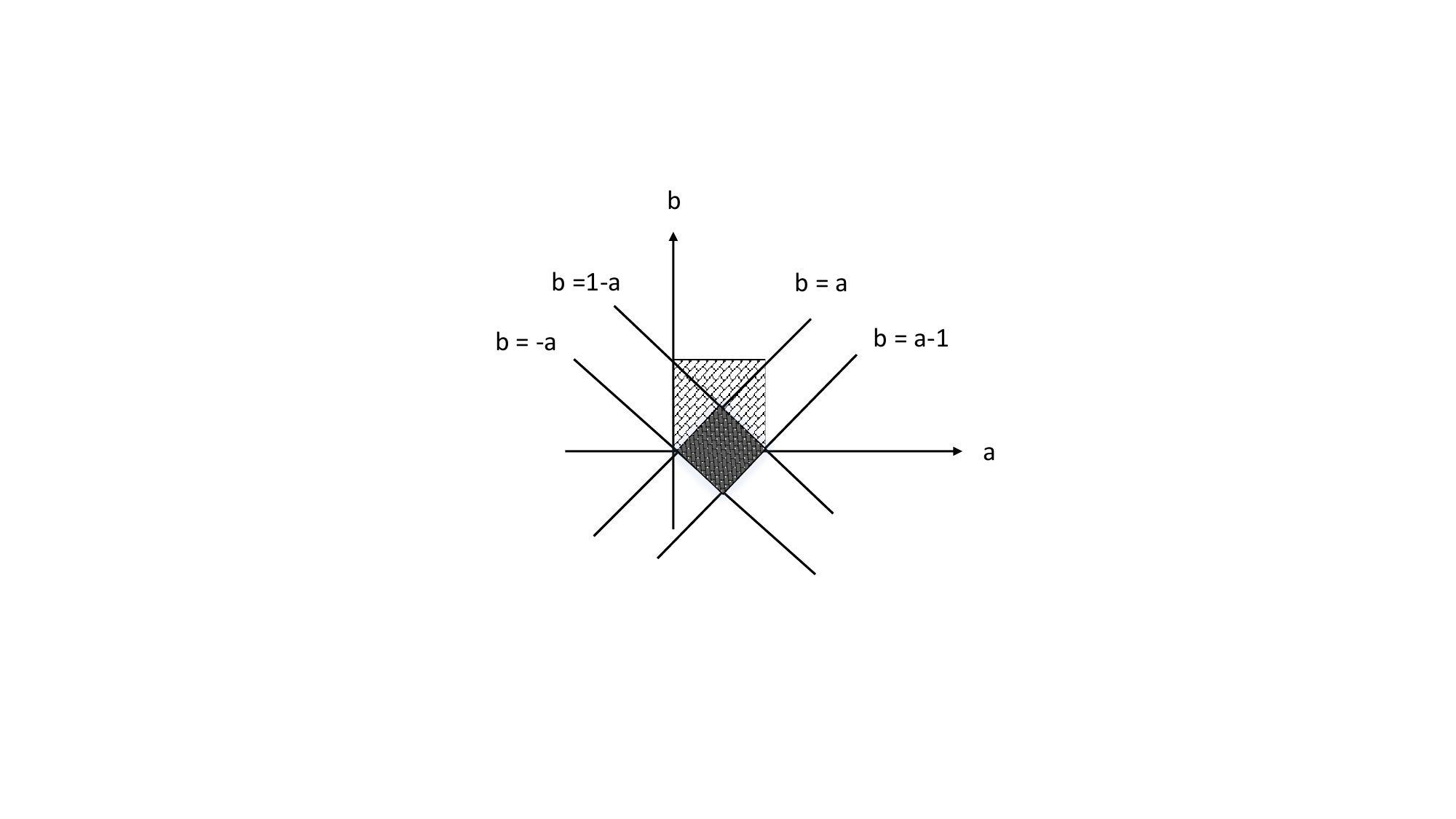}
        \vspace{-2cm}
        \caption{The set $\mathbb I$}
        \label{fig-1}
\end{figure}

In the setting of matrices, write $M=A+kB\in\mathbb{H}^{n\times n}$,
with $A$ and $B$ in $\mathbb{R}^{n\times n}$. Then $M$ is positive
semi-definite if and only if 
\[
-A\le B\le A.
\]
It follows in particular that both $A$ and $B$ are symmetric.

\subsection{Hyperbolic Probability}

In this section, we review definitions and results of \cite{MR3651492} where the probabilities
with  values in the hyperbolic numbers $\mathbb{H}$ are defined
and their properties are studied. It is shown there that the $\mathbb{H}$-valued
probabilities act just like the classical ($\mathbb{R}$-valued) probabilities
under natural additional conditions, and the corresponding probability
theory over $\mathbb{H}$ is established. This probability theory
of \cite{MR3651492} over $\mathbb{H}$ is extended to those over the scaled-hyperbolic
numbers $\mathbb{D}_{t}$ in \cite{MR4853153}, where $\mathbb{D}_{t}$
is the $t$-scaled hyperbolic numbers in the sense of \cite{MR4853153}
for all scales $t\in\mathbb{R}$. Scaled-hyperbolic numbers $\left\{ \mathbb{D}_{t}\right\} _{t\in\mathbb{R}}$
cover the complex numbers, 
\[
\mathbb{D}_{-1}=\mathbb{C}=\left\{ x+yi:x,y\in\mathbb{R},\;i^{2}=-1\right\} ,
\]
the dual numbers, 
\[
\mathbb{D}_{0}=\mathbf{D}=\left\{ x+yJ:x,y\in\mathbb{R},\;J^{2}=0\right\} ,
\]
and the hyperbolic numbers, 
\[
\mathbb{D}_{1}=\mathbb{H}=\left\{ x+yk:x,y\in\mathbb{R},\;k^{2}=1\right\} .
\]
See e.g. \cite{MR4853153, MR4920356}. By using the probability-theoretic structures of \cite{MR4920356}, we constructed $L^{2}$-space-like vector
spaces over the real field $\mathbb{R}$, and studied functional analysis
on them in \cite{alpay-cho-l2}. Motivated by the main constructions and results
of \cite{alpay-cho-l2}, we here focus on the hyperbolic numbers $\mathbb{H}$,
the $\mathbb{H}$-valued probabilities, and the corresponding statistical
analysis. To do that, in this section, we review basic concepts and
results of \cite{MR4853153}.\smallskip

\begin{definition}
Let $\left(X,\sigma\left(X\right)\right)$ be a measurable space equipped
with a set $X$ and a $\sigma$-algebra $\sigma\left(X\right)$. Suppose
\[
\mu:\sigma\left(X\right)\rightarrow\mathbb{H}
\]
is a $\mathbb{H}$-valued function satisfying the $\sigma$-additivity:
if $\left\{ X_{k}\right\} _{k\in\mathbb{N}}$ is a family of mutually
disjoint measurable subsets of $X$ in $\sigma\left(X\right)$, then
\[
\mu\left(\overset{\infty}{\underset{k=1}{\cup}}X_{k}\right)=\overset{\infty}{\underset{k=1}{\sum}}\mu\left(X_{k}\right),\;\;\;{in\;\;\;}\mathbb{H},
\]
where the convergence is in the topology of $\mathbb{H}$. Then this
function $\mu$ is called a $\mathbb{H}$-valued measure.
\end{definition}

If $f:\left(X,\sigma\left(X\right)\right)\rightarrow\mathbb{H}$ is
a $\mathbb{H}$-valued measurable function, then there exist $\mathbb{R}$-valued
measurable functions $f_{1},f_{2}:\left(X,\sigma\left(X\right)\right)\rightarrow\mathbb{R}$,
such that 
\[
f\left(x\right)=f_{1}\left(x\right)+f_{2}\left(x\right)k\in\mathbb{H},\;\;\;\forall x\in X,
\]
with 
\[
f\left(x\right)^{\circledast}=f_{1}\left(x\right)-f_{2}\left(x\right)k\in\mathbb{H},\;\;\;\forall x\in X,
\]
satisfying that 
\[
\mathrm{Re}\left(f\right)\overset{\mathrm{let}}{=}f_{1}=\frac{1}{2}\left(f+f^{\circledast}\right),\;\;\mathrm{and\;\;\mathrm{Im}}\left(f\right)\overset{\mathrm{let}}{=}f_{2}=\frac{1}{2}\left(f-f^{\circledast}\right).
\]
Now, let $\mu_{1}$ and $\mu_{2}$ be $\mathbb{R}$-valued measures
decomposed as
\[
\mu_{l}=\mu_{l}^{+}-\mu_{l}^{-},\;\;\;\; l=1,2,
\]
where $\mu_{l}^{+}$ and $\mu_{l}^{-}$ are non-negative measures on $\left(X,\sigma\left(X\right)\right)$, satisfying 
\[
\mu=\mu_{1}+k\mu_{2}=\left(\mu_{1}^{+}-\mu_{1}^{-}\right)+k\left(\mu_{2}^{+}-\mu_{2}^{-}\right).
\]
Then one can define the integral of a ${\mathbb H}$-valued measurable
function $f:\left(X,\sigma\left(X\right)\right)\rightarrow\mathbb{H}$
by

\medskip{}

$\;\;\;$$\int_{X}fd\mu=\int_{X}\left(\mathrm{Re}\left(f\right)+k\mathrm{Im}\left(f\right)\right)d\mu=\int_{X}f_{1}d\mu_{1}+k\int_{X}f_{2}d\mu_{2}$

\medskip{}

$\;\;\;\;\;\;\;\;\;\;\;\;\;\;\;\;$$=\int_{X}f_{1}d\left(\mu_{1}^{+}-\mu_{1}^{-}\right)+k\int_{X}f_{2}d\left(\mu_{2}^{+}-\mu_{2}^{-}\right)$

\medskip{}

$\;\;\;\;\;\;\;\;\;\;\;\;\;\;\;\;$$=\left(\int_{X}f_{1}d\mu_{1}^{+}-\int_{X}f_{1}d\mu_{1}^{-}\right)+k\left(\int_{X}f_{2}d\mu_{2}^{+}-\int_{X}f_{2}d\mu_{2}^{-}\right)$.

\begin{definition}
Let $\mu=\mu_{1}+k\mu_{2}$ be a $\mathbb{H}$-valued measure on a
measurable space $\left(X,\sigma\left(X\right)\right)$. This $\mathbb{H}$-valued
measure $\mu$ is said to be a $\mathbb{H}$-valued probability, if

(i) $\mu_{1}+\mu_{2}$ and $\mu_{1}-\mu_{2}$ are non-negative measures
on $\left(X,\sigma\left(X\right)\right)$, and

(ii) $\mu\left(X\right)\in\left\{ 1,e_{+},e_{-}\right\} $, i.e.,
$\mu\left(X\right)=1$, or $\mu\left(X\right)=e_{+}$, or $\mu\left(X\right)=e_{-}$.
\end{definition}

By definition, one can get the following result proven in {[}6{]}.
\begin{proposition}
Suppose $\mu:\left(X,\sigma\left(X\right)\right)\rightarrow{\mathbb H}$
is a ${\mathbb H}$-valued probability. Then:

$\mathrm{(1)}$ If $\mu\left(X\right)=1$, then $\left(\mu_{1}+\mu_{2}\right)\left(X\right)=1$,
and $\left(\mu_{1}-\mu_{2}\right)\left(X\right)=1$.

$\mathrm{(2)}$ If $\mu\left(X\right)=\mathbf{e}_{1}$, then $\left(\mu_{1}+\mu_{2}\right)\left(X\right)=1$,
and $\left(\mu_{1}-\mu_{2}\right)\left(X\right)=0$.

$\mathrm{(3)}$ If $\mu\left(X\right)=\mathbf{e}_{1}^{\circledast}$,
then $\left(\mu_{1}+\mu_{2}\right)\left(X\right)=0,$ and $\left(\mu_{1}-\mu_{2}\right)\left(X\right)=1$.
\end{proposition}

By the above proposition, if $\mu=\mu_{1}+k\mu_{2}$ is a $\mathbb{H}$-valued
probability, then at least one of the measures $\left(\mu_{1}+\mu_{2}\right)$ or $\left(\mu_{1}-\mu_{2}\right)$
is a classical ($\mathbb{R}$-valued non-negative) probability measure.

Now, let $f=f_{1}+kf_{2}=\left(f_{1}+f_{2}\right)e_{+}+\left(f_{1}-f_{2}\right)e_{-}$
be a $\mathbb{H}$-valued measurable function from a measurable space
$\left(X,\sigma\left(X\right)\right)$ to the hyperbolic numbers $\mathbb{H}$.
Then the integral of $f$ for $\mu$ is defined by 
\[
\int_{X}fd\mu\overset{\mathrm{def}}{=}\left(\int_{X}\left(f_{1}+f_{2}\right)d\left(\mu_{1}+\mu_{2}\right)\right)e_{+}+\left(\int_{X}\left(f_{1}-f_{2}\right)d\left(\mu_{1}-\mu_{2}\right)\right)e_{-}.
\]

The following proposition contains basic $\mathbb{H}$-valued-probabilistic properties shown in \cite{MR3651492}. The order relation has been defined in \eqref{2-10}.
\begin{proposition}
Let $\mu$ be the $\mathbb{H}$-valued probability on a measurable
space $\left(X,\sigma\left(X\right)\right)$.

\noindent $\mathrm{(1)}$ If $A\in\sigma\left(X\right)$, then $\mu\left(A\right)+\mu\left(A^{c}\right)=p$, with $p\in\left\{ 1,e_{+},e_{-}\right\} $.

\noindent $\mathrm{(2)}$ If $\phi$ is the empty set of $\sigma\left(X\right)$
in $X$, then $\mu\left(\phi\right)=0$.

\noindent $\mathrm{(3)}$ If $A\subseteq B$ in $\sigma\left(X\right)$,
then $\mu\left(A\right)\leq\mu\left(B\right)$ in $\mathbb{H}$.

\noindent $\mathrm{(4)}$ If $A,B\in\sigma\left(X\right)$, then 
\[
\mu\left(A\cup B\right)=\mu\left(A\right)+\mu\left(B\right)-\mu\left(A\cap B\right),
\]
and hence, 
\[
\mu\left(A\cup B\right)\leq\mu\left(A\right)+\mu\left(B\right),\;\;\;\mathrm{in\;\;\;}\mathbb{H}.
\]

\noindent $\mathrm{(5)}$ If $A_{1}\supseteq A_{2}\supseteq A_{3}\supseteq...,\;\mathrm{\textrm{in}\;}\sigma\left(X\right)$,
and $A=\overset{\infty}{\underset{k=1}{\cap}}A_{k}$ in $X$, then
\[
\underset{k\rightarrow\infty}{\mathrm{lim}}\mu\left(A_{k}\right)=\mu\left(A\right)=\mu\left(\overset{\infty}{\underset{k=1}{\cap}}A_{k}\right).
\]

\noindent $\mathrm{(6)}$ If $A_{1}\subseteq A_{2}\subseteq A_{3}\subseteq...$,
in $\sigma\left(X\right)$, and $A=\overset{\infty}{\underset{k=1}{\cup}}A_{k}$
in $\sigma\left(X\right)$, then 
\[
\underset{k\rightarrow\infty}{\mathrm{lim}}\mu\left(A_{k}\right)=\mu\left(A\right)=\mu\left(\overset{\infty}{\underset{k=1}{\cup}}A_{k}\right).
\]
\end{proposition}

Now, define a subset $\mathfrak{S}$ of $\mathbb{H}$ by 
\[
\mathfrak{S}\overset{\mathrm{def}}{=}\left\{ d=x+yk\in\mathbb{H}:x^{2}-y^{2}=0\right\} .
\]

\begin{definition}
Let $\left(X,\sigma\left(X\right),\mu\right)$ be a $\mathbb{H}$-valued-probability
space where 
\[
\mu=\mu_{1}+\mu_{2}k=\left(\mu_{1}+\mu_{2}\right)e_{+}+\left(\mu_{1}-\mu_{2}\right)e_{-}.
\]
For $A,B\in\sigma\left(X\right)$, define the conditional $\mathbb{H}$-valued
probability $\mu\left(A\mid B\right)$ by 
\[
  \mu\left(A\mid B\right)\overset{\mathrm{def}}{=}\frac{\left(\mu_{1}+\mu_{2}\right)\left(A\cap B\right)}{\left(\mu_{1}+\mu_{2}\right)\left(B\right)}e_{+}
  +\frac{\left(\mu_{1}-\mu_{2}\right)\left(A\cap B\right)}{\left(\mu_{1}-\mu_{2}\right)\left(B\right)}e_{-},
\]
where 
\[
\frac{\left(\mu_{1}+\mu_{2}\right)\left(A\cap B\right)}{\left(\mu_{1}+\mu_{2}\right)\left(B\right)}=\left(\mu_{1}+\mu_{2}\right)\left(A\mid B\right),
\]
and 
\[
\frac{\left(\mu_{1}-\mu_{2}\right)\left(A\cap B\right)}{\left(\mu_{1}-\mu_{2}\right)\left(B\right)}=\left(\mu_{1}-\mu_{2}\right)\left(A\mid B\right),
\]
are the usual $\mathbb{R}$-valued non-negative conditional measures
with axiomatization: 
\[
\mu\left(B\right)=0=0+0k\;\;\overset{\mathrm{axiom}}{\Longrightarrow}\;\;\mu\left(A\mid B\right)=\mu\left(A\right).
\]
\end{definition}

Using the above definition, one obtains the following result as in the usual probability theory.
\begin{proposition}
Let $\left(X,\sigma\left(X\right),\mu\right)$ be a $\mathbb{H}$-valued-probability
space. Then, for any $A,B\in\sigma\left(X\right)$, 
\[
\mu\left(A\cap B\right)=\mu\left(B\right)\mu\left(A\mid B\right).
\]
\end{proposition}

By the above proposition, one can have independence as in the usual
probability theory.
\begin{definition}
Let $\left(X,\sigma\left(X\right),\mu_{t}\right)$ be a $\mathbb H$-valued-probability
space, and $A,B\in\sigma\left(X\right)$.

(1) $A$ is said to be independent of $B$, if $\mu\left(A\mid B\right)=\mu\left(A\right)$.

(2) $B$ is said to be independent of $A$, if $\mu\left(B\mid A\right)=\mu\left(B\right)$.

(3) $A$ and $B$ are said to be mutually independent, if both $A$
is independent of $B$, and $B$ is independent of $A$.
\end{definition}

The above $\mathbb{H}$-valued probability theory is extended to $\mathbb{D}_{t}$-valued
probability theory for any fixed $t$-scaled hyperbolic numbers $\mathbb{D}_{t}$
of \cite{MR4853153} in \cite{MR4920356,alpay-cho-l2}, for all scales $t\in\mathbb{R}$.
Such an extension allows us to have not only $\mathbb{H}$-valued
probability theory, but also, $\mathbb{C}$-valued probability theory
and $\mathbf{D}$-valued probability theory, where $\mathbf{D}$ denotes the dual numbers.

\subsection{$L^{2}$-Spaces for Hyperbolic Probabilities}

In this subsection, we consider some functional-analytic structure induced
by our $\mathbb{H}$-valued probability spaces. This construction is needed to define integrals in Sections \ref{sec-4} and \ref{sec-5}, and
use Gelfand triple techniques for direct summands in Sections \ref{sec-sto} and \ref{deri}.
Let $\left(X,\sigma\left(X\right),\mu\right)$ be a $\mathbb{H}$-valued probability space satisfying $\mu\left(X\right)\in\left\{ 1,e_{+},e_{-}\right\} $,
and let $\chi_{S}:X\rightarrow\mathbb{H}$ be the characteristic function
of any arbitrarily fixed measurable subset $S\in\sigma\left(X\right)$,
i.e., 
\[
\chi_{S}\left(x\right)\overset{\mathrm{def}}{=}\left\{ \begin{array}{ccc}
1=1+0k &  & \mathrm{if\;}x\in S,\\
\\
0=0+0k &  & \mathrm{otherwise,}
\end{array}\right.
\]
for all $x\in X$. Define a family $\mathcal{F}_{X}$ by 
\[
\mathcal{F}_{X}=\left\{ f:X\rightarrow\mathbb{H}\left|f\textrm{ is a function}\right.\right\} ,
\]
and its subset $\mathcal{M}_{X}$ by 
\[
\mathcal{M}_{X}\overset{\mathrm{def}}{=}\left\{ \underset{S\in\sigma\left(X\right)}{\sum}d_{S}\chi_{S}\in\mathcal{F}_{X}:d_{S}\in\mathbb{H},\;\forall S\in\sigma\left(X\right)\right\} ,
\]
i.e., it is the collection of all measurable $\mathbb{H}$-valued simple
functions on the measurable space $\left(X,\sigma\left(X\right)\right)$.
Remark here that $\sum$ is a finite sum. By regarding all characteristic
functions $\left\{ \chi_{S}:S\in\sigma\left(X\right)\right\} $ generating
the measurable-functional family $\mathcal{M}_{X}$ as their images
$\left\{ 0,1\right\} \subset\mathbb{R}$ in $\mathbb{H}$, define
a morphism, 
\[
\left\Vert .\right\Vert _{\infty}:\mathcal{F}_{X}\rightarrow\mathbb{R},
\]
by 
\[
\left\Vert f\right\Vert _{\infty}\overset{\mathrm{def}}{=}\mathrm{sup}\left\{ \left\Vert f\left(x\right)\right\Vert :x\in X\right\} ,\;\;\;\forall f\in\mathcal{F}_{X},
\]
where $\left\Vert .\right\Vert $ in the right-hand side is the $\mathbb{R}$-norm
on the hyperbolic numbers $\mathbb{H}$ introduced in Section 2.2
above. i.e., 
\[
\left\Vert u+vk\right\Vert =\sqrt{u^{2}+v^{2}},\;\;\;\forall u+vk\in\mathbb{H}.
\]
Then, since $\left\Vert .\right\Vert $ is a well-defined complete
$\mathbb{R}$-norm on $\mathbb{H}$, this morphism $\left\Vert .\right\Vert _{\infty}$
forms a well-defined $\mathbb{R}$-norm on $\mathcal{F}_{X}$, and
hence, a $\mathbb{R}$-norm on $\mathcal{M}_{X}$. Define the $\left\Vert .\right\Vert _{\infty}$-norm-topology
completion $\mathfrak{M}_{X}$ of $\mathcal{M}_{X}$ in $\mathcal{F}_{X}$,
i.e., 
\[
\mathfrak{M}_{X}\overset{\mathrm{def}}{=}\overline{\mathcal{M}_{X}}^{\left\Vert .\right\Vert _{\infty}},\;\;\;\;\mathrm{in\;\;\;}\mathcal{F}_{X},
\]
where $\overline{Y}^{\left\Vert .\right\Vert _{\infty}}$ means the
$\left\Vert .\right\Vert _{\infty}$-norm-topology closure of a subset
$Y$ of $\mathcal{F}_{X}$. Then 
\[
\mathfrak{M}_{X}=\left\{ \underset{S\in\sigma\left(X\right)}{\sum}d_{S}\chi_{S}\in\mathcal{F}_{X}:d_{S}\in\mathbb{H},\;\forall S\in\sigma\left(X\right)\right\} ,
\]
where $\sum$ now is an infinite sum (or, the limit of finite sums
under $\left\Vert .\right\Vert _{\infty}$). Note that this topological
space $\mathfrak{M}_{X}$ forms a well-defined vector space over the
real field $\mathbb{R}$ (in short, a $\mathbb{R}$-vector space)
equipped with the usual functional addition and the $\mathbb{R}$-scalar
product. i.e., $\mathfrak{M}_{X}$ forms a Banach (vector) space over
$\mathbb{R}$ (in short, a $\mathbb{R}$-Banach space). As usual in
measure theory, if $f=\underset{S\in\sigma\left(X\right)}{\sum}d_{S}\chi_{S}\in\mathfrak{M}_{X}$,
then its integral is well-defined to be 
\[
\int_{X}fd\mu=\underset{S\in\sigma\left(X\right)}{\sum}d_{S}\mu\left(S\right)\in\mathbb{H}.
\]
i.e., $\mathfrak{M}_{X}$ forms a $\mathbb{R}$-Banach space of all
$\mu$-integrable $\mathbb{H}$-valued measurable functions on $\left(X,\sigma\left(X\right)\right)$.

Now, define a form, 
\[
\left[,\right]:\mathfrak{M}_{X}\times\mathfrak{M}_{X}\rightarrow\mathbb{H},
\]
by 
\[
\left[f,g\right]\overset{\mathrm{def}}{=}\int_{X}fg^{\circledast}d\mu,\;\;\;\forall f,g\in\mathfrak{M}_{X},
\]
where if $g=\underset{Y\in\sigma\left(X\right)}{\sum}q_{Y}\chi_{Y}\in\mathfrak{M}_{X}$
with $q_{Y}\in\mathbb{H}$, then 
\[
g^{\circledast}=\underset{Y\in\sigma\left(X\right)}{\sum}q_{Y}^{\circledast}\chi_{Y}\in\mathfrak{M}_{X},
\]
where $\circledast$ is the involution on the hyperbolic numbers $\mathbb{H}$. Suppose 
\[
f=\underset{S\in\sigma\left(X\right)}{\sum}d_{S}\chi_{S},\;\;g=\underset{Y\in\sigma\left(X\right)}{\sum}q_{Y}\chi_{Y}\in\mathfrak{M}_{X}.
\]
Then

\bigskip{}

$\;\;\;\;$$\left[f,g\right]=\int_{X}fg^{\circledast}d\mu=\int_{X}\left(\underset{\left(S,Y\right)\in\sigma\left(X\right)^{2}}{\sum}d_{S}q_{Y}^{\circledast}\chi_{S\cap Y}\right)d\mu$

\bigskip{}

$\;\;\;\;\;\;\;\;\;\;\;\;\;\;\;\;\;\;$$=\underset{\left(S,Y\right)\in\sigma\left(X\right)^{2}}{\sum}d_{S}q_{Y}^{\circledast}\mu\left(S\cap Y\right)\in\mathbb{H}$,

\bigskip{}

\noindent implying that 
\[
\left[f,f\right]=\underset{\left(S_{1},S_{2}\right)\in\sigma\left(X\right)^{2}}{\sum}d_{S_{1}}d_{S_{2}}^{\circledast}\mu\left(S_{1}\cap S_{2}\right)\in\mathbb{H},
\]
satisfying the boundedness, 
\[
\left\Vert \left[f,f\right]\right\Vert \leq\left\Vert \underset{\left(S_{1},S_{2}\right)\in\sigma\left(X\right)^{2}}{\sum}d_{S_{1}}d_{S_{2}}^{\circledast}\right\Vert <\left\Vert f\right\Vert _{\infty}^{2}<\infty,
\]
since $f\in\mathfrak{M}_{X}$. By regarding the hyperbolic numbers
$\mathbb{H}$ as $1$-scaled hyperbolics $\mathbb{D}_{1}$ of \cite{MR4920356,alpay-cho-l2}, this form $\left[,\right]$ is a $\mathbb{D}_{1}$-valued
inner product (or, a $\mathbb{H}$-indefinite inner product) on $\mathfrak{M}_{X}$,
satisfying (i) 
\[
\left[f,g\right]\in\mathbb{H},\;\;\;\;\;\forall f,g\in\mathfrak{M}_{X},
\]
(ii) for all $d_{1},d_{2},d_{3}\in\mathbb{H}$ and $f_{1},f_{2},f_{3}\in\mathfrak{M}_{X}$,
\[
\left[d_{1}f_{1}+d_{2}f_{2},f_{3}\right]=d_{1}\left[f_{1},f_{3}\right]+d_{2}\left[f_{2},f_{2}\right]\in\mathbb{H},
\]
and 
\[
\left[f_{1},d_{2}f_{2}+d_{3}f_{3}\right]=d_{2}^{\circledast}\left[f_{1},f_{2}\right]+d_{3}^{\circledast}\left[f_{1},f_{3}\right]\in\mathbb{H},
\]
(iii) one also has 
\[
\left[f_{1},f_{2}\right]=\left[f_{2},f_{1}\right]^{\circledast},\;\;\;\;\forall f_{1},f_{2}\in\mathfrak{M}_{X},
\]
(iv) for any fixed $f\in\mathfrak{M}_{X}$, 
\[
\left[f,g\right]=0=0+0k\in\mathbb{H},\;\;\forall g\in\mathfrak{M}_{X}\Longrightarrow f=0=0+0k\in\mathfrak{M}_{X}.
\]
Now, let's define a (topological) subspace, 
\[
L_{X}^{2}\overset{\mathrm{denote}}{=}L^{2}\left(X,\sigma\left(X\right),\mu\right)\overset{\mathrm{def}}{=}\left\{ f\in\mathfrak{M}_{X}:\int_{X}\left\Vert f\right\Vert ^{2}d\mu<\infty\right\} ,
\]
of $\mathfrak{M}_{X}$, equipped with the inherited $\mathbb{H}$-indefinite
inner product $\left[,\right]$ on $L_{X}^{2}$. Then it forms a well-determined
$\mathbb{R}$-Banach (vector) subspace of $\mathfrak{M}_{X}$. By
definition, one can define a new norm $\left\Vert .\right\Vert _{2}$
on $L_{X}^{2}$ by 
\[
\left\Vert f\right\Vert _{2}\overset{\mathrm{def}}{=}\sqrt{\int_{X}\left\Vert f\right\Vert ^{2}d\mu},\;\;\;\;\;\forall f\in L_{X}^{2}.
\]
Then the inherited functional addition from that on $\mathfrak{M}_{X}$
is closed on $L_{X}^{2}$ under this new norm $\left\Vert .\right\Vert _{2}$
because of the Minkowski's inequality, and the $\mathbb{R}$-scalar
product is closed on $L_{X}^{2}$ canonically. i.e., this $\mathbb{R}$-normed
space $\left(L_{X}^{2},\left\Vert .\right\Vert _{2}\right)$ forms
a new $\mathbb{R}$-Banach space since $\left\Vert .\right\Vert _{2}$
is complete on $L_{X}^{2}$ by the completeness of the $\mathbb{R}$-norm
$\left\Vert .\right\Vert $ on the hyperbolic numbers $\mathbb{H}$.
From below, we understand $L_{X}^{2}$ as a $\mathbb{R}$-Banach space
equipped with the $\mathbb{R}$-norm $\left\Vert .\right\Vert _{2}$.

Now, let $f=f_{1}+f_{2}k\in L_{X}^{2}$ with its $\mathbb{R}$-valued
functions $f_{1}$ and $f_{2}$, equivalently, 
\[
f=\left(f_{1}+f_{2}\right)e_{+}+\left(f_{1}-f_{2}\right)e_{-}\overset{\mathrm{denote}}{=}f_{+}e_{+}+f_{-}e_{-},
\]
where $f_{+}$ and $f_{-}$ are $\mathbb{R}$-valued measurable functions
whenever our given $\mathbb{H}$-valued measure $\mu$ satisfies 
\[
\mu=\mu_{+}e_{+}+\mu_{-}e_{-},
\]
where $\mu_{+}$ and $\mu_{-}$ are bounded non-negative $\mathbb{R}$-valued
measures. Since 
\[
\mathbb{H}=\mathbb{R}e_{+}\oplus\mathbb{R}e_{-},
\]
under 
\[
e_{+}^{2}=e_{+},\;\;\;e_{-}^{2}=e_{-},\;\;\;\mathrm{and\;\;\;}e_{+}e_{-}=0=e_{-}e_{+},
\]
in $\mathbb{H}$, if $f=f_{+}e_{+}+f_{-}e_{-}\in L_{X}^{2}$, then
\[
f_{+}\in L^{2}\left(X,\sigma\left(X\right),\mu_{+}\right),\;\;\;\mathrm{and\;\;\;}f_{-}\in L^{2}\left(X,\sigma\left(X\right),\mu_{-}\right),
\]
where $L^{2}\left(X,\sigma\left(X\right),\mu_{e}\right)$ are the
classical $L^{2}$-space which are $\mathbb{R}$-Hilbert spaces equipped
with their $\mathbb{R}$-definite inner products, 
\[
\left\langle f,g\right\rangle _{2}=\int_{X}fgd\mu_{e},\;\;\;\;\forall f,g\in L^{2}\left(X,\sigma\left(X\right),\mu_{e}\right),
\]
for all $e\in\left\{ \pm\right\} $. Conversely, if $f=f_{+}e_{+}+f_{-}e_{-}\in\mathfrak{M}_{X}$,
and 
\[
f_{+}\in L^{2}\left(X,\sigma\left(X\right),\mu_{+}\right),\;\;\;f_{-}\in L^{2}\left(X,\sigma\left(X\right),\mu_{-}\right),
\]
then 
\[
f\in L_{X}^{2}=L^{2}\left(X,\sigma\left(X\right),\mu\right),
\]
whenever 
\[
\mu=\mu_{+}e_{+}+\mu_{-}e_{-}\textrm{ is a }\mathbb{H}\textrm{-valued measure}.
\]
Therefore, without loss of generality, one can get that 
\[
L_{X}^{2}=L^{2}\left(X,\sigma\left(X\right),\mu\right)=L^{2}\left(X,\sigma\left(X\right),\mu_{+}\right)e_{+}\oplus L^{2}\left(X,\sigma\left(X\right),\mu_{-}\right)e_{-}.
\]

\begin{theorem}
Suppose $L_{X}^{2}=L^{2}\left(X,\sigma\left(X\right),\mu\right)$
is the $\mathbb{R}$-Banach space on a $\mathbb{H}$-valued probability
space $\left(X,\sigma\left(X\right),\mu\right)$ with a $\mathbb{H}$-valued
probability $\mu=\mu_{+}e_{+}+\mu_{-}e_{-}$. Then 
\[
L_{X}^{2}=L^{2}\left(X,\sigma\left(X\right),\mu_{+}\right)e_{+}\oplus L^{2}\left(X,\sigma\left(X\right),\mu_{-}\right)e_{-},
\]
where $L^{2}\left(X,\sigma\left(X\right),\mu_{e}\right)$ are the
usual $L^{2}$-Hilbert spaces equipped with its $\mathbb{R}$-definite
inner product $\left\langle ,\right\rangle _{2}$ for all $e\in\left\{ \pm\right\} $.
\end{theorem}

\begin{proof}
Let $f=f_{+}e_{+}+f_{-}e_{-}\in L_{X}^{2}$, and let $f_{1}=f_{+}e_{+}$
and $f_{2}=f_{-}e_{-}$ in $L_{X}^{2}$. Then 
\[
\left[f_{1},f_{2}\right]=\int_{X}f_{1}f_{2}^{\circledast}d\mu=\int_{X}f_{+}f_{-}e_{+}e_{-}d\mu=\int_{X}\left(0+0k\right)d\mu=0,
\]
in $\mathbb{H}$. i.e., $f_{1}$ and $f_{2}$ are orthogonal under
$\left[,\right]$ in $L_{X}^{2}$.
\end{proof}
By the above theorem, without loss of generality, one can represent every element $f\in L_{X}^{2}$ as $f_{+}e_{+}+f_{-}e_{-}\in\mathfrak{M}_{X}$
with 
\[
f_{+}\in L^{2}\left(X,\sigma\left(X\right),\mu_{+}\right),\;\;\mathrm{and\;\;}f_{-}\in L^{2}\left(X,\sigma\left(X\right),\mu_{-}\right),
\]
whenever $\mu=\mu_{+}e_{+}+\mu_{-}e_{-}$. So, if $f=f_{+}e_{+}+f_{-}e_{-}$
and $g=g_{+}e_{+}+g_{-}e_{-}$ are in $L_{X}^{2}$, then 
\[
\left[f,g\right]=\int_{X}fg^{\circledast}d\mu=\left(\int_{X}f_{+}g_{+}d\mu_{+}\right)e_{+}-\left(\int_{X}f_{-}g_{-}d\mu_{-}\right)e_{-},
\]
i.e., 
\[
\left[f,g\right]=\left\langle f_{+},g_{+}\right\rangle _{2}e_{+}+\left\langle f_{-},g_{-}\right\rangle _{2}e_{-}.
\]
in $\mathbb{H}$, since $e_{+}^{2}=e_{+}$, $e_{-}^{2}=e_{-}$ and
$e_{+}e_{-}=0=e_{-}e_{+}$ in $\mathbb{H}$.

\bigskip{}

\begin{remark} As we have seen in \cite{MR4920356,alpay-cho-l2},
the above decomposition of $L_{X}^{2}$ holds only when $\mu$ is
a $\mathbb{D}_{1}=\mathbb{H}$-valued probability, (or, $\mathbb{D}_{t}$-valued
probability with a positive scale $t>0$ up to isomorphisms). If a
given scale $t\leq0$, then $\mathbb{D}_{t}$-valued probability $\mu_{t}$
does not give such a decomposition of $L^{2}\left(X,\sigma\left(X\right),\mu_{t}\right)$.
\end{remark}
\bigskip{}

\begin{corollary}
Let $\left(X,\sigma\left(X\right),\mu\right)$ be a $\mathbb{H}$-valued
probability space, and let $L_{X}^{2}=L^{2}\left(X,\sigma\left(X\right),\mu\right)$
be the $\mathbb{R}$-Banach space equipped with the $\mathbb{H}$-indefinite
inner product $\left[,\right]$. If 
\[
f=f_{+}e_{+}+f_{-}e_{-},\;\;g=g_{+}e_{+}+g_{-}e_{-}\in L_{X}^{2},
\]
then 
\[
\left[f,g\right]=\left\langle f_{+},g_{+}\right\rangle _{2}e_{+}+\left\langle f_{-},g_{-}\right\rangle _{2}e_{-}\in\mathbb{H},
\]
where $\left\langle ,\right\rangle _{2}$ is the $\mathbb{R}$-definite
inner product on the usual $L^{2}$-spaces $L^{2}\left(X,\sigma\left(X\right),\mu_{e}\right)$
for all $e\in\left\{ \pm\right\} $. So, this $\mathbb{R}$-Banach
space $L_{X}^{2}$ is $\mathbb{H}$-indefinite-inner-product $\mathbb{H}$-bi-module
in the sense of \cite{MR5000850}.
\end{corollary}

\begin{proof}
The formula for the $\mathbb{H}$-indefinite inner product is shown
by the very above paragraph up to the decomposition of $L_{X}^{2}$.
It is not hard to check that 
\[
d\in\mathbb{H},\;\;f\in L_{X}^{2}\Longrightarrow df\in L_{X}^{2},
\]
because if $d=d_{+}e_{+}+d_{-}e_{-}\in\mathbb{H}$ with $d_{+},d_{-}\in\mathbb{R}$,
and $f=f_{+}e_{+}+f_{-}e_{-}\in L_{X}^{2}$ with $f_{e}\in L^{2}\left(X,\sigma\left(X\right),\mu_{e}\right)$
for all $e\in\left\{ \pm\right\} $, then 
\[
df=\left(d_{+}f_{+}\right)e_{+}+\left(d_{-}f_{-}\right)e_{+}\in L_{X}^{2},
\]
and similarly, 
\[
fd=df\in L_{X}^{2}.
\]
And hence, $L_{X}^{2}$ is a well-defined bi-module over $\mathbb{H}$,
i.e., a $\mathbb{H}$-bi-module.
\end{proof}

\section{Positive definite kernels}

\setcounter{equation}{0} \label{sec-3} We consider here positive
definite kernels of a very specific form, namely (with $p=\lambda_1e_++\lambda_2e_-$ and $q=\mu_1e_++\mu_2e_-$)
\begin{equation}
K(p,q)=U\begin{pmatrix}K_{1}(\lambda_{1},\mu_{1}) & 0\\
0 & K_{2}(\lambda_{2},\mu_{2})
\end{pmatrix}U,\label{K1-K2}
\end{equation}
where $K_{1}$ and $K_{2}$ are real-valued and positive definite
on some subset $S$ of the real numbers, and $p,q\in\Omega$, where
\[
\Omega=\left\{ \lambda_{1}e_{+}+\lambda_{2}e_{-}\,\,\,;\,\,\,\lambda_{1},\lambda_{2}\in S\,\right\} .
\]

Let $\mathfrak{H}(K_{1})$ and $\mathfrak{H}(K_{2})$ be the reproducing
kernel Hilbert spaces on the real numbers with reproducing kernels
$K_{1}$ and $K_{2}$ respectively. By the Zaremba-Bergman formula
we write 
\begin{eqnarray}
K_{1}(\lambda_{1},\mu_{1}) & = & \sum_{a\in A}f_{a}^{(1)}(\lambda_{1})f_{a}^{(1)}(\mu_{1})\\
K_{2}(\lambda_{2},\mu_{2}) & = & \sum_{b\in B}f_{b}^{(2)}(\lambda_{2})f_{b}^{(2)}(\mu_{2})
\end{eqnarray}
where $(f_{a}^{(1)})_{a\in A}$ denotes any orthonormal basis of $\mathfrak{H}(K_{1})$
and similarly, $(f_{b}^{(2)})_{b\in B}$ denotes any orthonormal basis
of $\mathfrak{H}(K_{2})$. One has

\[
\begin{split}K(p,q) & =\left(\sum_{a\in A}f_{a}^{(1)}(\lambda_{1})f_{a}^{(1)}(\mu_{1})\right)U\begin{pmatrix}1 & 0\\
0 & 0
\end{pmatrix}U+\\
 & \hspace{5mm}+\left(\sum_{b\in B}f_{b}^{(2)}(\lambda_{2})f_{a}^{(2)}(\mu_{2})\right)U\begin{pmatrix}0 & 0\\
0 & 1
\end{pmatrix}U\\
 & =\left(\sum_{a\in A}f_{a}^{(1)}(\lambda_{1})f_{a}^{(1)}(\mu_{1})\right)\frac{1}{2}\begin{pmatrix}1 & 1\\
1 & 1
\end{pmatrix}+\\
 & \hspace{5mm}+\left(\sum_{b\in B}f_{b}^{(2)}(\lambda_{2})f_{a}^{(2)}(\mu_{2})\right)\frac{1}{2}\begin{pmatrix}1 & -1\\
-1 & 1
\end{pmatrix}\\
 & =\sum_{a\in A}F_{a}^{(1)}(\lambda_{1})F_{a}^{(1)}(\mu_{1})+\sum_{b\in B}F_{b}^{(2)}(\lambda_{2})F_{b}^{(2)}(\mu_{2})
\end{split}
\]
where 
\[
F_{a}^{(1)}(\lambda)=f_{a}^{(1)}(\lambda)\frac{1}{2}\begin{pmatrix}1 & 1\\
1 & 1
\end{pmatrix}\quad{\rm and}\quad F_{a}^{(2)}(\lambda)=f_{b}^{(2)}(\lambda)\frac{1}{2}\begin{pmatrix}1 & -1\\
-1 & 1
\end{pmatrix}.
\]
It follows that the reproducing kernel Hilbert space $\mathfrak{H}(K)$
(on the real numbers) with reproducing kernel \eqref{K1-K2} can be
described as the set of $\mathbb{H}$-valued functions of the form
\begin{equation}
F(p)=\sum_{a\in A}F_{a}^{(1)}(\lambda_{1})A_{a}+\sum_{b\in B}F_{b}^{(2)}(\lambda_{2})B_{b},\quad A_{a},B_{b}\in\mathbb{R},
\end{equation}
with $\mathbb{H}$ norm 
\[
[F,F]=\left(\sum_{a\in A}A_{a}^{2}\right)e_{+}+\left(\sum_{b\in B}B_{b}^{2}\right)e_{-}
\]
and norm 
\begin{equation}
\|F\|^{2}={\rm Tr}\,[F,F]=\sum_{a\in A}A_{a}^{2}+\sum_{b\in B}B_{b}^{2}.
\end{equation}

We note that we do not consider here the case where $K_{1}$ or $K_{2}$
are complex-valued. Then, one has to use bi-complex numbers rather
than hyperbolic numbers. To keep unity in the paper, we have chosen
not to use bi-complex numbers here.

\section{$\mathbb{H}$-valued probabilities and random variables}
\label{sec-4}
\setcounter{equation}{0} In this section we follow \cite{MR3651492}
and review a number of facts on $\mathbb{H}$-valued probabilities
(note that the algebra of hyperbolic numbers is denoted by $\mathbb{D}$
in \cite{MR3651492}). Let $(\Omega,\mathcal{C})$ be a measure space
and let $P_{1}$ and $P_{2}$ be two probability measures on $(\Omega,\mathcal{C})$.
We build the $\mathbb{H}$-valued probability 
\begin{equation}
P_{\mathbb{H}}=U\begin{pmatrix}P_{1} & 0\\
0 & P_{2}
\end{pmatrix}U.\label{6789!!!}
\end{equation}
We note that $P_{\mathbb{H}}$ takes values in the unit interval $\mathbb{I}$
of the hyperbolic numbers. It is readily seen that $P_{\mathbb{H}}$
satisfies the properties of an $\mathbb{H}$-valued probability, as
presented in \cite{MR3651492}. In that paper the cases $P_{1}=0$
or $P_{2}=0$ are allowed, but will not be considered here.

If $\left(X,\sigma\left(X\right),\mu\right)$ is a $\mathbb{H}$-valued
probability space with its $\mathbb{H}$-valued probability, 
\[
\mu=\mu_{+}e_{+}+\mu_{-}e_{-},
\]
where $\mu_{+}$ and $\mu_{-}$ are the usual bounded non-negative
($\mathbb{R}$-valued) measures on the measurable space $\left(X,\sigma\left(X\right)\right)$,
satisfying 
\[
\mu\left(X\right)=1=e_{+}+e_{-},\;\;\mathrm{or\;\;}\mu\left(X\right)=e_{+},\;\;\mathrm{or\;\;}\mu\left(X\right)=e_{-}.
\]
So, it is possible that 
\[
\mathrm{either\;\;}\mu=\mu_{+}e_{+}+Oe_{-},\;\;\mathrm{or\;\;}\mu=Oe_{+}+\mu_{-}e_{-},
\]
if and only if 
\[
\mu_{+}\left(X\right)=1,\;\;\;\;\mathrm{respectively,\;\;\;}\mu_{-}\left(X\right)=1,
\]
where $O:\sigma\left(X\right)\rightarrow\mathbb{R}$ is the ($\mathbb{R}$-valued)
zero measure, i.e., $O\left(S\right)=0$, for all $S\in\sigma\left(X\right)$,
which is a well-defined bounded non-negative $\mathbb{R}$-valued measure,
if and only if 
\[
\mu_{+}\textrm{ is a classical }\mathbb{R}\textrm{-valued probability on }\left(X,\sigma\left(X\right)\right),\textrm{ and }\mu_{-}=O,
\]
respectively, 
\[
\mu_{+}=O,\;\mathrm{and\;}\mu_{-}\textrm{ is a classical }\mathbb{R}\textrm{-valued probability on }\left(X,\sigma\left(X\right)\right),
\]
for a given $\mathbb{H}$-valued probability $\mu=\mu_{+}e_{+}+\mu_{-}e_{-}$.
e.g., see Section 8 of \cite{MR4920356} by replacing the general scale
$t\in\mathbb{R}$ to the fixed scale $1$.

If we restrict our interests to the cases where either 
\[
\mu_{+}\textrm{ is a classical probability, and }\mu_{-}=O,
\]
or 
\[
\mu_{+}=O,\;\textrm{and }\mu_{-}\textrm{ is a classical probability,}
\]
making $\mu=\mu_{+}e_{+}+\mu_{-}e_{-}$ be a $\mathbb{H}$-valued
probability on $\left(X,\sigma\left(X\right)\right)$, then one obtains
that 
\[
L^{2}\left(X,\sigma\left(X\right)\mu\right)=L^{2}\left(X,\sigma\left(X\right),\mu_{+}\right)e_{+}\oplus L^{2}\left(X,\sigma\left(X\right),O\right)e_{-},
\]
respectively, 
\[
L^{2}\left(X,\sigma\left(X\right)\mu\right)=L^{2}\left(X,\sigma\left(X\right),O\right)e_{+}\oplus L^{2}\left(X,\sigma\left(X\right),\mu_{-}\right)e_{-},
\]
because 
\[
L^{2}\left(X,\sigma\left(X\right),\mu\right)=L^{2}\left(X,\sigma\left(X\right),\mu_{+}\right)e_{+}\oplus L^{2}\left(X,\sigma\left(X\right),\mu_{-}\right)e_{-}.
\]
So, if $f=f_{+}e_{+}+f_{-}e_{-}\in L^{2}\left(X,\sigma\left(X\right),\mu\right)$,
and if either 
\[
\mu_{+}\textrm{ is a classical probability, and }\mu_{-}=O,
\]
or 
\[
\mu_{+}=O,\;\textrm{and }\mu_{-}\textrm{ is a classical probability,}
\]
then 
\[
\int_{X}fd\mu=\left(\int_{X}f_{+}d\mu_{+}\right)e_{+}+0e_{-}=\left(\int_{X}f_{+}d\mu_{+}\right)e_{+},
\]
respectively, 
\[
\int_{X}fd\mu=0e_{+}+\left(\int_{X}f_{-}d\mu_{-}\right)e_{-}=\left(\int_{X}f_{-}d\mu_{-}\right)e_{-},
\]
in the hyperbolic numbers $\mathbb{H}$.
\begin{corollary}
Let $\left(X,\sigma\left(X\right),\mu\right)$ be a $\mathbb{H}$-valued
probability space with its $\mathbb{H}$-valued probability $\mu=\mu_{+}e_{+}+\mu_{-}e_{-}$,
where $\mu_{e}$ are bounded non-negative $\mathbb{R}$-valued measures
for all $e\in\left\{ \pm\right\} $, and let $L^{2}\left(X,\sigma\left(X\right),\mu\right)$
be the corresponding $\mathbb{R}$-Banach space equipped with its
$\mathbb{H}$-indefinite inner product $\left[,\right]$. Suppose
either 
\[
\mu_{+}\textrm{ is a classical probability, and }\mu_{-}=O,
\]
or 
\[
\mu_{+}=O,\;\textrm{and }\mu_{-}\textrm{ is a classical probability,}
\]
and hence, 
\[
\mu=\mu_{+}e_{+}+Oe_{-},\;\;\mathrm{respectively,\;\;}\mu=Oe_{+}+\mu_{-}e_{-}.
\]
Then
\[
L^{2}\left(X,\sigma\left(X\right),\mu\right)=L^{2}\left(X,\sigma\left(X\right),\mu_{+}\right)e_{+}\oplus L^{2}\left(X,\sigma\left(X\right),O\right)e_{-},
\]
respectively,
\[
L^{2}\left(X,\sigma\left(X\right),\mu\right)=L^{2}\left(X,\sigma\left(X\right),O\right)e_{+}\oplus L^{2}\left(X,\sigma\left(X\right),\mu_{-}\right)e_{-}.
\]
and hence, if $f=f_{+}e_{+}+f_{-}e_{-}$ and $g=g_{+}e_{+}+g_{-}e_{-}$
in $L^{2}\left(X,\sigma\left(X\right),\mu\right)$, then 
\[
\left[f,g\right]=\left\langle f_{+},g_{+}\right\rangle _{2}e_{+}+0e_{-}=\left(\int_{X}f_{+}g_{+}d\mu_{+}\right)e_{+}\in\mathbb{H},
\]
respectively,
\[
\left[f,g\right]=0e_{+}+\left\langle f_{-},g_{-}\right\rangle _{2}e_{-}=\left(\int_{X}f_{-}g_{-}d\mu_{-}\right)e_{-}\in\mathbb{H}.
\]
\end{corollary}

\begin{proof}
The decomposition of $L^{2}\left(X,\sigma\left(X\right),\mu\right)$
is proven by the above theorem. So, as a special case, if either 
\[
\mu_{+}\textrm{ is a classical probability, and }\mu_{-}=O,
\]
or 
\[
\mu_{+}=O,\;\textrm{and }\mu_{-}\textrm{ is a classical probability,}
\]
then 
\[
L^{2}\left(X,\sigma\left(X\right),\mu\right)=L^{2}\left(X,\sigma\left(X\right),\mu_{+}\right)e_{+}\oplus L^{2}\left(X,\sigma\left(X\right),O\right)e_{-},
\]
respectively, 
\[
L^{2}\left(X,\sigma\left(X\right),\mu\right)=L^{2}\left(X,\sigma\left(X\right),O\right)e_{+}\oplus L^{2}\left(X,\sigma\left(X\right),\mu_{-}\right)e_{-}.
\]
Under the zero measure $O$, we trivially have that 
\[
\int_{X}g\:dO=0\in\mathbb{R},\;\;\;\forall g\in L^{2}\left(X,\sigma\left(X\right),O\right),
\]
and hence, 
\[
\left\langle g_{1},g_{2}\right\rangle _{2}=0,\;\;\;\;\;\forall g_{1},g_{2}\in L^{2}\left(X,\sigma\left(X\right),O\right).
\]
Recall that if $f=f_{+}e_{+}+f_{-}e_{-},\;g=g_{+}e_{+}+g_{-}e_{-}\in L^{2}\left(X,\sigma\left(X\right),\mu\right)$,
then 
\[
\left[f,g\right]=\left\langle f_{+},g_{+}\right\rangle _{2}e_{+}+\left\langle f_{-},g_{-}\right\rangle _{2}e_{-}\in\mathbb{H},
\]
by the above corollary. Therefore, the above $\mathbb{H}$-indefinite
inner product computations hold case-by-case.
\end{proof}
As we checked above, if we fix an arbitrary $\mathbb{H}$-valued probability
$\mu=\mu_{+}e_{+}+\mu_{-}e_{-}$ on a measurable space $\left(X,\sigma\left(X\right)\right)$,
then it means that $\mu$ satisfies the one of the following three
cases; 
\[
\mu_{+},\;\mu_{-}\textrm{ are usual probabilities satisfying }\mu\left(X\right)=e_{+}+e_{-}=1,
\]
or 
\[
\mu_{+}\textrm{ is a usual probability, and }\mu_{-}=O,\textrm{ satisfying }\mu\left(X\right)=e_{+},
\]
or 
\[
\mu_{+}=O,\textrm{ and }\mu_{-}\textrm{ is a usual probability, satisfying }\mu\left(X\right)=e_{-}.
\]

\bigskip{}

\noindent $\mathbf{Observation.}$ Suppose $\mu=\mu_{+}e_{+}+\mu_{-}e_{-}$
is a $\mathbb{H}$-valued probability satisfying
\[
\mathrm{either\;\;}\mu_{-}=O,\;\;\;\mathrm{or\;\;\;}\mu_{+}=O.
\]
Then, by the proof of the above corollary, if $f=f_{+}e_{+}+f_{-}e_{-}\in L^{2}\left(X,\sigma\left(X\right),\mu\right)$,
then the analytic information of $f$ is determined by 
\[
\mathrm{those\;of\;}f_{+},\;\;\mathrm{respectively,\;\;those\;of\;}f_{-}.
\]
Equivalently, the analytic properties of $f$ is characterized by
those of
\[
f_{+}\in L^{2}\left(X,\sigma\left(X\right),\mu_{+}\right),\;\;\mathrm{respectively,\;\;}f_{-}\in L^{2}\left(X,\sigma\left(X\right),\mu_{-}\right),
\]
under the decomposition, 
\[
L^{2}\left(X,\sigma\left(X\right),\mu\right)=L^{2}\left(X,\sigma\left(X\right),\mu_{+}\right)e_{+}\oplus L^{2}\left(X,\sigma\left(X\right),\mu_{-}\right)e_{-}.
\]
It means that the classical theories for $L^{2}\left(X,\sigma\left(X\right),\mu_{e}\right)$
dictates those for $L^{2}\left(X,\sigma\left(X\right),\mu\right)$
for $e\in\left\{ \pm\right\} $ on 
\[
\mathbb{H}=\mathbb{R}e_{+}\oplus\mathbb{R}e_{-}.
\]

\bigskip{}

\noindent $\mathbf{Assumption.}$ Let $\mu=\mu_{+}e_{+}+\mu_{-}e_{-}$
be a $\mathbb{H}$-valued probability on a measurable space $\left(X,\sigma\left(X\right)\right)$.
By the above observation, in the following text, we concentrate on
a case where both $\mu_{+}$ and $\mu_{-}$ are usual probabilities, making 
\[
\mu\left(X\right)=e_{+}+e_{-}=1=1+0k\in\mathbb{H}.
\]

\begin{definition} The $\mathbb{H}$-valued function defined on $\Omega$,
with idempotent decomposition $T=T_{1}e_{+}+T_{2}e_{-}$ is called
a random variable if both $T_{1}$ and $T_{2}$ are random variables,
i.e are measurable functions. \end{definition}

\begin{definition} The random variable $T$ has a probability density
\[
g(p)=U\begin{pmatrix}g_{1}(\lambda_{1},\mu_{1}) & 0\\
0 & g_{2}(\lambda_{2},\mu_{2})
\end{pmatrix}U
\]
if it holds that 
\begin{equation}
\begin{split}P_{\mathbb{H}}(T\le p)=\\
 & \hspace{-1cm}=U\begin{pmatrix}\iint_{(-\infty,\lambda_{1}]\times(-\infty,\lambda_{2}]}g_{1}(u,v)dudv & 0\\
0 & \iint_{(-\infty,\lambda_{1}]\times(-\infty,\lambda_{2}]}g_{2}(u,v)dudv
\end{pmatrix}U,
\end{split}
\end{equation}
where $P_{\mathbb H}$ is defined by  \eqref{6789!!!}.
\end{definition}

\begin{definition} We set
\begin{equation}
E_{\mathbb{H}}Z=U\begin{pmatrix}\iint_{(-\infty,\lambda_{1}]\times(-\infty,\lambda_{2}]}ug_{1}(u,v)dudv & 0\\
0 & \iint_{(-\infty,\lambda_{1}]\times(-\infty,\lambda_{2}]}vg_{2}(u,v)dudv
\end{pmatrix}U.
\end{equation}
\end{definition} 

\begin{proposition} Let $(\Omega,\mathcal{C})$ and $P_{1}$ and
$P_{2}$ as above, and let $T$ be a $\mathbb{H}$-valued random variable,
with idempotent decomposition 
\begin{equation}
T=U\begin{pmatrix}T_{1} & 0\\
0 & T_{2}
\end{pmatrix}U,
\end{equation}
and $p\in\mathbb{H}$ with idempotent decomposition as in \eqref{p}.
Furthermore,

\begin{equation}
E_{\mathbb H}T=U\begin{pmatrix}E_{1}T_{1} & 0\\
0 & E_{2}T_{2}
\end{pmatrix}U.
\end{equation}

Assume that $T_{1}$ and $T_{2}$ are independent with respect to both
$P_{1}$ and $P_{2}$. We have 
\begin{equation}
\begin{split}P_{\mathbb{H}}(T\le p) & =U\begin{pmatrix}P_{1}(T_{1}\le\lambda_{1})P_{1}(T_{2}\le\lambda_{2}) & 0\\
0 & P_{2}(T_{1}\le\lambda_{1})P_{2}(T_{2}\le\lambda_{2})
\end{pmatrix}U\\
 & =P_{\mathbb{H}}(T_{1}\le\lambda_{1})P_{\mathbb{H}}(T_{2}\le\lambda_{2}).
\end{split}
\label{1234}
\end{equation}
\end{proposition}
\begin{proof}
The claim follows from 
\begin{equation}
\left\{ T\le p\right\} =\left\{ T_{1}\le\lambda_{1}\right\} \cap\left\{ T_{2}\le\lambda_{2}\right\} ,
\end{equation}
and in view of the independence hypothesis, 
\begin{equation}
P_{1}(\left\{ T_{1}\le\lambda_{1}\right\} \cap\left\{ T_{2}\le\lambda_{2}\right\} )=P_{1}(\left\{ T_{1}\le\lambda_{1}\right\} )P_{1}(\left\{ T_{2}\le\lambda_{2}\right\} )
\end{equation}
and 
\begin{equation}
P_{2}(\left\{ T_{1}\le\lambda_{1}\right\} \cap\left\{ T_{2}\le\lambda_{2}\right\} )=P_{2}(\left\{ T_{1}\le\lambda_{1}\right\} )P_{2}(\left\{ T_{2}\le\lambda_{2}\right\} ).
\end{equation}

Then, from \eqref{1234} 
\begin{equation}
\begin{split}P_{\mathbb{H}}(Z\le q) & =U\begin{pmatrix}P_{1}(Z_{1}\le\lambda_{1})P_{1}(Z_{2}\le\lambda_{2}) & 0\\
0 & P_{2}(Z_{1}\le\lambda_{1})P_{2}(Z_{2}\le\lambda_{2})
\end{pmatrix}U\\
 & =U\begin{pmatrix}P_{1}(Z_{1}\le\lambda_{1}) & 0\\
0 & P_{2}(Z_{1}\le\lambda_{1})
\end{pmatrix}UU\begin{pmatrix}P_{1}(Z_{2}\le\lambda_{2}) & 0\\
0 & P_{2}(Z_{2}\le\lambda_{2})
\end{pmatrix}U\\
 & =P_{\mathbb{H}}(Z_{1}\le\lambda_{1})P_{\mathbb{H}}(Z_{2}\le\lambda_{2}).
\end{split}
\end{equation}
\end{proof}
Based on the above definitions we can now introduce $\mathbb{H}$-valued
normal Gaussian variables as follows:

\begin{definition} Let $Z_{1}$ and $Z_{2}$ be two (not necessarily independent) $N(0,1)$
random variables defined on $(\Omega,\mathcal{C})$. The $\mathbb{H}$-valued random variable 
\begin{equation}
U\begin{pmatrix}Z_{1} & 0\\
0 & Z_{2}
\end{pmatrix}U,
\end{equation}
is called the $\mathbb{H}$-valued normal Gaussian variable associated
to $Z_{1}$ and $Z_{2}$. \end{definition}

\begin{definition} A second-order $\mathbb{H}$-valued stochastic
process indexed by $\mathbb{H}$ is a map $p\mapsto X(p)$ from $\mathbb{H}$
into ${L}_{2}(\Omega,\mathcal{C},P_{\mathbb{H}})$. Its covariance
function is defined to be 
\begin{equation}
E_{\mathbb H}(X(p)X(q))=\int_{\Omega}X(p,\w)X(q,\w)dP_{\mathbb{H}}(\w).
\end{equation}
With 
\begin{equation}
X(p)=U\begin{pmatrix}X_{1}(p) & 0\\
0 & X_{2}(p)
\end{pmatrix}U,\label{repre}
\end{equation}
we have 
\begin{equation}
E_{\mathbb H}(X(p)X(q))=U\begin{pmatrix}E_{1}(X_{1}(p)X_{1}(q)) & 0\\
0 & E_{2}(X_{2}(p)X_{2}(q))
\end{pmatrix}U.\label{456654}
\end{equation}
\end{definition}

\begin{definition} A $\mathbb{H}$-valued stochastic process is Gaussian
if, in the representation \eqref{repre} both $X_{1}(p)$ and $X_{2}(p)$
are both Gaussian processes. \end{definition}

The following is the counterpart in the setting of hyperbolic probabilities
of Lo\`eve's theorem.

\begin{theorem} The covariance of a second-order $\mathbb{H}$-valued
stochastic process is positive definite. Conversely, every $\mathbb{H}$-valued
function positive definite on $\mathbb{H}$ is the covariance function
of a second-order $\mathbb{H}$-valued stochastic process, which can
be chosen to be Gaussian. \end{theorem}
\begin{proof}
In the notation of Section \ref{sec-3} we write 
\begin{eqnarray}
K_{1}(p,q) & = & \sum_{a\in A}f_{a}^{(1)}(p)f_{a}^{(1)}(q)\\
K_{2}(p,q) & = & \sum_{b\in B}f_{b}^{(2)}(p)f_{b}^{(2)}(q)
\end{eqnarray}
where now the functions depend on $p$ and not only on $\lambda_{1}$
or $\lambda_{2}$. As in the introduction, consider now 
\[
\Omega_{1}=(\mathbb{R},\frac{1}{\sqrt{2\pi}}e^{-\frac{-x^{2}}{2}}dx)^{A}\quad{\rm and}\quad\Omega_{2}=(\mathbb{R},\frac{1}{\sqrt{2\pi}}e^{-\frac{-x^{2}}{2}}dx)^{B}.
\]
$\Omega_{1}$ contains a family $(Z_{a}^{(1)})_{a\in A}$ of independent
$N(0,1)$ random variables and similarly for $\Omega_{2}$. We set
\[
X_{j}(p)=\sum_{a\in A}f_{a}^{(j)}(p)Z_{a}^{(j)}(\w),\quad j=1,2,
\]
and 
\[
X(p)=U\begin{pmatrix}X_{1}(p) & 0\\
0 & X_{2}(p)
\end{pmatrix}U.
\]
We consider the $\mathbb{H}$-valued probability measure $P=U\begin{pmatrix}P_{1} & 0\\
0 & P_{2}
\end{pmatrix}U$ on the probability space $(\Omega_{1}\times\Omega_{2})$ endowed
with the product sigma-algebra . Then 
\[
E_{\mathbb H}X(p)X(q)=K(p,q).
\]
Furthermore, one can write 
\[
X(p)=\sum_{\substack{a\in A\\
b\in B
}
}U\begin{pmatrix}f_{a}^{(1)}(p) & 0\\
0 & f_{b}^{(2)}(p)
\end{pmatrix}UU\begin{pmatrix}Z_{a}^{(1)}(\w) & 0\\
0 & Z_{b}^{(2)}(\w)
\end{pmatrix}U
\]
where the $\mathbb{H}$-valued random variables 
\[
Z_{a,b}(\w)=U\begin{pmatrix}Z_{a}^{(1)}(\w) & 0\\
0 & Z_{b}^{(2)}(\w)
\end{pmatrix}U,\quad a\in A,\,\ b\in B
\]
are independent Gaussian variables. 
\end{proof}
In the sequel, we focus on processes $X$ for which, in the decomposition
\eqref{repre} we have that $X_{1}$ is a function of $\lambda_{1}$
and $X_{2}$ is a function of $\lambda_{2}$. Then \eqref{456654}
is of the form \eqref{K1-K2}.

\section{The hyperbolic white noise space}
\setcounter{equation}{0}
\label{sec-5}
In order to construct the derivative of the Brownian motion, we use Hida's white noise space approach to construct
a specific probability space $(\Omega,\mathcal{A},P)$, where the
$\mathbb{H}$-valued Brownian motion will be defined, and an associated
Gelfand triple, to build the associated $\mathbb{H}$-valued white
noise. The main aspects of Hida's theory can be found in \cite{MR562914,Hi90,MR1387829},
and can be seen as a stochastic version of the Gelfand triple 
\[
(\mathcal{S}_{\mathbb{R}},{L}^{2}(\mathbb{R},dx,\mathbb{R}),\mathcal{S}_{\mathbb{R}}^{\prime}).
\]
In this section we content ourselves to describing the main aspects
of the theory needed in this paper.\smallskip{}

The construction of the probability space involves a deep theorem
of functional analysis, namely the Bochner-Minlos theorem. Denoting
by $\mathcal{S}_{\mathbb{R}}$ the nuclear Fr\'echet space of Schwartz
test functions, and by $\mathcal{S}_{\mathbb{R}}^{\prime}$ its dual
(the space of real-valued tempered distributions), the Bochner-Minlos
theorem asserts that there exists a probability measure on the cylinder
sigma-algebra of $\mathcal{S}_{\mathbb{R}}^{\prime}$ such that 
\begin{equation}
e^{-\frac{\|s\|_{2}^{2}}{2}}=\int_{\mathcal{S}_{\mathbb{R}}^{\prime}}e^{i\langle s,\w\rangle}dP(\w),\quad s\in\mathcal{S}_{\mathbb{R}},\label{Piso-123456}
\end{equation}
and where $\|s\|_{2}$ denotes the norm in the Lebesgue space. It
follows in particular from this equation that the random variable
\begin{equation}
Q_{s}(\w)=\langle s,\w\rangle
\end{equation}
is a centered Gaussian with covariance $\|s\|_{2}^{2}$. In the notation
of the first paragraph of the introduction (recall that we denote
by $\zeta_{1},\zeta_{2},\ldots$ denote the normalized Hermite functions),
we define 
\begin{equation}
Z_{n}=Q_{\zeta_{n}},\quad n=1,2,\ldots,\label{znn}
\end{equation}
The map $s\mapsto Q_{s}$ extends to an isometry from ${L}^{2}(\mathbb{R},dx)$
into ${L}^{2}(\mathcal{S}_{\mathbb{R}}^{\prime},\mathcal{C},P)$,
which we will denote by the same symbol.\smallskip{}

Let $\ell$ denote the set of infinite sequences $(\alpha_{1},\alpha_{2},\ldots)$
indexed by $\mathbb{N}$, and with values in $\mathbb{N}_{0}$, and
for which $\alpha_{j}\not=0$ for at most a finite number of indices
and let $h_{0},h_{1},\ldots$ denote the probabilists Hermite polynomials.
Then, (\cite[Theorem 2.2.3 p. 24]{new_sde}) the functions 
\begin{equation}
  \label{5-4}
H_{\alpha}(\w)=\prod_{n=1}^{\infty}h_{\alpha_{n}}(Q_{\zeta_{n}})
\end{equation}
form an orthogonal basis of ${L}^{2}(\mathcal{S}_{\mathbb{R}}^{\prime},\mathcal{C},P)$,
and 
\begin{equation}
  \label{5-5}
\langle H_{\alpha},H_{\beta}\rangle_{P}=\delta_{\alpha,\beta}\alpha!.
\end{equation}

In linear system theory convolution of the coefficients indices (also
called Cauchy product; see \cite{MR51:583}) allows multiplication
when point-wise multiplication is not possible or not the right tool;
see e.g. and \cite{am_ieot} and \cite{MR4762055} for two examples.
The present setting is no exception and the point-wise product is replaced
by the convolution on $\ell$, here called Wick product, defined by
\begin{equation}
  H_{\alpha}\star H_{\beta}=H_{\alpha+\beta}
\end{equation}
The Wick product of two elements of ${L}^{2}(\mathcal{S}_{\mathbb{R}}^{\prime},\mathcal{C},P)$
need not belong to ${L}^{2}(\mathcal{S}_{\mathbb{R}}^{\prime},\mathcal{C},P)$,
and we embed the latter in a topological algebra of a special type
(called strong algebra in \cite{MR3404695}) and in which the Wick
product is stable. The first such algebra originates with the work
of Vage \cite{vage96b,vage96} and was defined by Kondratiev; see
\cite{new_sde}. The general theory of such algebras was later developed
by the first named author with G. Salomon; see \cite{MR3029153,MR3404695}.\smallskip{}

\begin{remark}
  \label{two-BM}
The same analysis with the function
\begin{equation}
  \label{s-sigma}
e^{-\frac{\int_{\mathbb R}|\widehat{s}(u)|^2d\sigma(u)}{2}}
\end{equation}
where $d\sigma$ is a positive measure satisfying
\begin{equation}
\int_{\mathbb R}\frac{d\sigma(u)}{u^2+1}<\infty
\label{m-u-2-3}
\end{equation}
and where
\[
  \widehat{s}(u)=\int_{\mathbb R}e^{-iut}s(t)dt
  \]
  denotes the Fourier transform, will lead to another probability $P_\sigma$ on the cylinder algebra, and another family of independent
  $N(0,1)$ variables, say $W_0,W_1,\ldots$.\smallskip

  In more details, first remark that, in view of \eqref{m-u-2-3}, for any Schwartz function the quantity
  \[
\|s\|_{\sigma}=\left(    \int_{\mathbb R}|\widehat{s}(u)|^2d\sigma(u)\right)^{\frac{1}{2}}
\]
is finite. For this to hold it would have been enough to consider the condition
\begin{equation}
\int_{\mathbb R}\frac{d\sigma(u)}{(u^2+1)^N}<\infty
\label{m-u-2-3-5}
\end{equation}
for some $N\in\mathbb N$. For $N>1$ the convergence of the integrals in  \eqref{9-6}-\eqref{9-9} is not guaranteed anymore. See \cite{ajnfao,MR3402823}
for further information in the classical case when \eqref{m-u-2-3-5} is in force.\smallskip

We denote by $\stackrel{\circ}{\mathcal H    (\sigma)}$ the space $\mathcal S_{\mathbb R}$ endowed with the pre-Hilbert
space structure defined by \eqref{m-u-2-3}.
The set $\stackrel{\circ}{\mathcal H(\sigma)}$  need not be closed (for instance, with $d\sigma(u)=\frac{du}{u^2+1}$,
the function $\widehat{s}\equiv 1$ will correspond to the Dirac distribution).\smallskip

The  map \eqref{s-sigma} is continuous in the Fr\'echet topology of $\mathcal S$, see \cite[Theorem 5.2, (5.4)]{MR2793121}.
  We can thus apply the Bochner-Minlos theorem to \eqref{s-sigma} and write
  \begin{equation}
 e^{-\frac{\int_{\mathbb R}|\widehat{s}(u)|^2d\sigma(u)}{2}}=\int_{\mathcal S^\prime_{\mathbb R}} e^{i\langle w,s\rangle}dP_\sigma(w),
    \end{equation}
  for a uniquely defined cylinder measure $P_\sigma$. The map $Q_s$ is now a $N(0,1)$ centered Gaussian  with variance
  $\int_{\mathbb R}|\widehat{s}(u)|^2d\sigma(u)$ and the application $s\mapsto Q_s$ is an isometry from
  $\stackrel{\circ}{\mathcal H(\sigma)}$ into $L^2(\mathcal S^\prime_{\mathbb R},\mathcal C,P_\sigma)$:
\[
\langle s_1,s_2\rangle_\sigma=\int_{\mathbb R}\overline{\widehat{s_1}(u)}\widehat{s_2}(u)d\sigma(u), \quad s_1,s_2\in\mathcal S_{\mathbb R}.
\]
  
  The map $s\mapsto Q_s$ extends to an everywhere defined isometry from the closure  $\mathcal H(\sigma)$ of $\stackrel{\circ}{\mathcal H(\sigma)}$ into
  $L^2(\mathcal S^\prime_{\mathbb R},\mathcal C,P_\sigma)$. 
To proceed one chooses an orthonormal system $g_0,g_1,\ldots$ of $\mathcal H(\sigma)$. The $g_j$ need not belong to $\mathcal S_{\mathbb R}$.
(In the classical case one can take the $g_j$ to be the normalized Hermite functions).
Then by \cite[Theorem 3.21 p. 29]{MR99f:60082}, the corresponding functions \eqref{5-4} with the $\zeta_n$ replaced by the $g_n$ form an
orthogonal basis which also satisfies \eqref{5-5}). After normalization by $\sqrt{\alpha!}$ we get an orthonormal basis of the corresponding probability
space $L^2(\mathcal S^\prime_{\mathbb R},\mathcal C,P_\sigma)$.
\end{remark}

\begin{corollary} Let $\sigma$ and $\mu$ be two positive measures on the real line, satisfying \eqref{m-u-2-3}, and for which the corresponding Hilbert
  spaces $\mathcal H(\sigma)$ and $\mathcal H(\mu)$ are countable. Let respectively $g_0,g_1,\ldots$ and $h_0,h_1,\ldots$ be orthonormal basis of
  these spaces. We denote by $Z_n=Q_{g_n}$ and $W_n=Q_{h_n}$ in
  $L^2(\mathcal S^\prime_{\mathbb R},\mathcal C,P_\sigma)$ and $L^2(\mathcal S^\prime_{\mathbb R},\mathcal C,P_\mu)$.
 \label{cor-78}
\end{corollary}

 We get on $\mathcal S^\prime_{\mathbb R}$ two probability space structures, in which $Z_0,Z_1,\ldots$ and $W_0,W_1,\ldots$ are respectively
  orthonormal families. The Gaussian variables $Z_n$ and $W_m$ need not be independent with respect to either probability.

\begin{definition}
  Let $\sigma$ and $\mu$ be two positive measures on the real line satisfying \eqref{m-u-2}.
The space $L^2(\mathcal S_{\mathbb R}^\prime,\mathcal C,P_{\mathbb H})$ with $P_1=P_{\sigma}$ and $P_2=P_{\mu}$ is called hyperbolic white noise space.
  \end{definition}
\section{The $\mathbb{H}$-valued Brownian motion}

\setcounter{equation}{0}

\begin{proposition} Let $p\wedge q$ and $|p-q|_{\mathbb H}$ be defined by \eqref{min-sup} and \eqref{p-q} respectively. The function 
\[
p\wedge q,\quad p,q\in\mathbb{H}_{+}
\]
is positive definite on $\mathbb{H}_{+}$ and the function
\begin{equation}
\frac{1}{2}\left(|p|_{\mathbb{H}}+|q|_{\mathbb{H}}-|p-q|_{\mathbb{H}}\right),\quad p,q\in\mathbb{H}\label{wertyuio}
\end{equation}
is positive definite on $\mathbb{H}$. \end{proposition} 
\begin{proof}
This follows from \eqref{K1-K2} and the corresponding fact for the
functions $\lambda_{1}\wedge\mu_{1}$ and $|\lambda_{1}|+|\mu_{1}|-|\lambda_{1}-\mu_{1}|$. 
\end{proof}
\begin{definition} The stochastic process on the probability space
$(\Omega,\mathcal{C},P)$ where $P$ is $\mathbb{H}$-valued is called
a $\mathbb{H}$-valued Brownian motion if it is Gaussian and if 
\begin{equation}
E_{\mathbb H}(B_{p}B_{q})=p\wedge q,\quad p,q\in\mathbb{H}_{+}
\end{equation}
and more generally, using \eqref{p-q} 
\begin{equation}
E_{\mathbb H}(B_{p}B_{q})=\frac{1}{2}\left(|p|_{\mathbb{H}}+|q|_{\mathbb{H}}-|p-q|_{\mathbb{H}}\right),\quad p,q\in\mathbb{H}.
\end{equation}
\end{definition}

The existence and construction of $B_{p}$ follow from the previous
section. More precisely:

\begin{theorem} Let, as in Corollary \ref{cor-78}, be given two probability measures on the cylinder sigma-algebra of $\mathcal S^\prime_{\mathbb R}$, and let $Z_1,Z_2,\ldots$ and
  $W_1,W_2,\ldots$ be the corresponding sequences of pairwise orthogonal Gaussian variables. Define (with $p=\lambda_1e_++\lambda_2e_-$)
\begin{equation}
B_{p}(\w)=\sum_{n=0}^{\infty}U\begin{pmatrix}\left(\int_{0}^{\lambda_{1}}\zeta_{n}(u)du\right)Z_{n}(\w) & 0\\
0 & \left(\int_{0}^{\lambda_{2}}\zeta_{n}(u)du\right)W_{n}(\w)
\end{pmatrix}U.
\end{equation}
Then, $(B_{p})$ 
is $\mathbb{H}$-valued Gaussian and has covariance
function \eqref{wertyuio}. \end{theorem}

\begin{proof}
The result follows from \eqref{repre} and \eqref{456654} using the construction in Corollary \ref{cor-78}.
  \end{proof}

The above arguments can be extended to define the fractional $\mathbb{H}$-valued
Brownian motion, which will be the Gaussian stochastic process with
covariance function 
\begin{equation}
K_{H}(p,q)=V_{H}\left\{ |p|_{\mathbb{H}}^{2H}+|q|_{\mathbb{H}}^{2H}-|p-q|_{\mathbb{H}}^{2H}\right\} \label{K-H-t-s}
\end{equation}
where $V_H$ is given by \eqref{V-H}. This is done in Section \ref{sec-fbm}.
\section{Stochastic Gelfand triples}
\setcounter{equation}{0}
\label{sec-sto}
We denote by $\ell$ the set of sequence of integers
\[
  \alpha=(\alpha_0,\alpha_1,\ldots)
\]
for which at most a finite number of entries are different from $0$.\smallskip

The Kondratiev space of stochastic test functions is equal to 
\[
\mathcal{K}=\quad\bigcap_{k\in\mathbb{N}}{\mathcal{G}}_{k}
\]
where for $k\in\mathbb{N}_{0}$, $\mathcal{G}_{k}$ is the Hilbert
space of series of the form 
\[
f(\omega)=\sum_{\alpha\in\ell}{c_{\alpha}H_{\alpha}(\omega)},\quad{with\,\,norm}\quad n_{k}(f)=\left(\sum_{\alpha\in\ell}(\alpha!)^{2}|c_{\alpha}|^{2}(2\mathbb{N})^{k\alpha}\right)^{1/2}<\infty,
\]
with $(2{\mathbb{N}})^{\alpha}=2^{\alpha_{1}}(2\times2)^{\alpha_{2}}(2\times3)^{\alpha_{3}}\cdots$.\smallskip

We denote by ${\mathcal{H}}_{k}$ the Hilbert space of formal series
of the form 
\[
  f(\omega)=\sum_{\alpha\in\ell}{c_{\alpha}H_{\alpha}(\omega)},\quad{with\,\,norm}\quad||f||_{-k}=\left(\sum_{\alpha\in\ell}{|c_{\alpha}|^{2}
      (2\mathbb{N})^{-k\alpha}}\right)^{1/2}<\infty.
\]
For example,  with 
\[
  e_n=(0,0,\ldots, \underbrace{\fbox{1}}_{n-th\,\, place},0,\ldots),\quad n=0,1,2,\ldots,
  \]
the random variables
\begin{equation}
  \label{zn}
  Z_n=H_{e_n}
\end{equation}
are independent $N(0,1)$ variables, and in particular
\begin{equation}
  \label{Hen}
\|H_{e_n}\|_{\mathcal H_{-k}}^2=(2n)^{-2}=\frac{1}{2^k(n+1)^k}
  \end{equation}

The Kondratiev space of stochastic distributions is the inductive
limit 
\[
\mathcal{K}^{\prime}=\quad\bigcup_{k\in{\mathbb{N}}}{{\mathcal{H}}_{k}}.
\]

$\mathcal{K}^{\prime}$ is the dual of $\mathcal{K}$ and allows to
define the Gelfand triple $(\mathcal{K},{L}^{2}(\mathcal{S}_{\mathbb{R}}^{\prime},\mathcal{C},P),\mathcal{K}^{\prime})$.\smallskip

The Kondratiev space is a topological algebra, which possesses a sequence of inequalities on the norms, introduced by Vage \cite{vage96}; see Theorem
\ref{dallas}. Other spaces with similar inequalities, called strong algebras, were introduced in \cite{MR3038506,MR3029153,  MR3404695}.\smallskip{}

The space $\mathcal{K}^{\prime}$ is not metrizable, but the injection
from $\mathcal{H}_{k}$ into $\mathcal{H}_{k+1}$ has a finite trace
for every $k\in\mathbb{N}$ (this implies in particular that the Fréchet
space $\mathcal{K}=\mathcal{K}^{\prime\prime}$ is nuclear, and hence
$\mathcal{K}^{\prime}$ is also nuclear, and in particular perfect
in the terminology of Gelfand and Shilov. One has then the following
property when working with sequences (see \cite[p. 57]{GS2_english}
for more details):

\begin{proposition} \label{prop} A sequence of elements $(x_{n})_{n\in\mathbb{N}}$
in $\mathcal{K}^{\prime}$ converges to $x\in\mathcal{K}^{\prime}$ in the strong topology if and only if there exists $p$ such that $x_{n}\in\mathcal{H}_{p}$
from a certain rank, $x\in\mathcal{H}_{p}$, and 
\[
\lim_{n\rightarrow\infty}\|x_{n}-x\|_{p}=0.
\]
\end{proposition}

This result is helpful to study functions defined on $[0,\infty)$.
Indeed, if $x(t):[0,\infty)\rightarrow\mathcal{K}^{\prime}$, then
since $[0,\infty)$ is a metric space, the existence of a limit is
equivalent to the existence of a limit in terms of sequences. This
property is used in the following section to computing the derivative
of the $\mathbb{H}$-valued Brownian motion. We also note:

\begin{proposition}
  The inductive topology and the strong topology coincide on $\mathcal K^\prime$, meaning that the strong dual of $\mathcal K$ is the inductive limit of the spaces
  $\mathcal H_k$.
\end{proposition}

See \cite[(iv) p. 215 and the references therein]{MR3404695}.\smallskip

The following result will not be used here but is mentioned because it is the key to develop stochastic calculus in the present setting. The inequality is called
Vage inequality, see \cite[Proposition 3.3.2 p. 129]{new_sde} and \cite[Proposition 2.3.3 p. 35]{new_sde} for the finiteness of the coefficient appearing
in \eqref{vage45}, and originates with the work of Vage; see \cite{vage96b,vage96}.

\begin{theorem} (see \cite{MR3029153,MR3404695}) With the convolution of coefficients as product, if $f\in\mathcal{H}_{p}$ and $g\in\mathcal{H}_{q}$
with $p-q\ge2$, then $f\star g\in\mathcal{H}_{-p}$ and there exists
a positive $A(p-q)$, depending only on $p$ and $q$, such that 
\begin{equation}
  \label{vage45}
\|f\star g\|_{-p}\le A(p-q)\|f\|_{-p}\cdot\|g\|_{-q},
\end{equation}
where
\begin{equation}
A(p-q)=\left(\sum_{\alpha\in\ell}(2\mathbb N)^{(q-p)\alpha}\right)^{1/2}.
  \end{equation}
\label{dallas} \end{theorem}

A similar construction holds {\sl verbatim} for the function \eqref{s-sigma}, using \cite[Theorem 3.21 p. 29]{MR99f:60082}.

\section{Derivatives}
\label{deri}
\setcounter{equation}{0}
Recall that $N_t(w)$ introduced in Section \ref{sec-int} (or e.g., \cite{MR562914}) is understood as  a stochastic distribution
meaning that    $dB_t / dt$ is a topological limit in an underlying locally convex vector space of distributions.

\begin{theorem} The sum 
\begin{equation}
N_{p}(\w)=\sum_{n=0}^{\infty}U\begin{pmatrix}\zeta_{n}(\lambda_{1})Z_{n}(\w) & 0\\
0 & \zeta_{n}(\lambda_{2})W_{n}(\w)
\end{pmatrix}U.
\end{equation}
belongs to $\mathcal{K}^{\prime}$. It is the derivative of $B_{p}$
in the topology of $\mathcal{K}^{\prime}$, in the sense that 
\begin{equation}
\lim_{\substack{h\rightarrow0\\
h^{-1}\,\,{exists}
}
}\frac{B_{p+h}-B_{p}}{h}=N_{p}.
\end{equation}
\end{theorem} 
\begin{proof}
To compute the derivative one needs to compute the derivatives of
the two underlying real-valued Brownian motion, and this is done as
in for instance \cite{aal2,aal3}, using the tools and framework presented
in the previous section.  Since we aim this paper to at least three audiences (hypercomplex analysis, probability and infinite dimensional analysis) we recall the details in the
classical case. We have
\[
B_{\lambda_1}(w)=\sum_{n=0}^\infty\left(\int_0^{\lambda_1}\zeta_n(u)du\right)Z_n(w)
  \]
  and proceed in a number of steps.\\

  STEP 1: {\sl The formal series $N_{\lambda_1}(w)=\sum_{n=0}^\infty\zeta_n(u) Z_n(w)$ belongs to $\mathcal H_{-2}$.}\smallskip

  The normalized Hermite functions are uniformly bounded in modulus on the real line; see \cite{thangavelu1993lectures}. The claim then follows from $A(2)<\infty$.
  See \cite[Proposition 3.3.2 p. 129]{new_sde}.\\

STEP 2: {\sl
Let $C$ be such that
\[
  |\zeta_n(x)|\le C<\infty, \quad x\in\mathbb R,\,\, n\in \mathbb N_0.
\]
Then,
\begin{equation}
|\zeta_n^\prime(x)|\le C(|x|+\sqrt{2(n+1)}).
  \label{bound}
  \end{equation}
}
This follows directly from the raising relation for the normalized Hermite functions,
  \[
    \zeta^\prime_n(x)=x\zeta_n(x)-\sqrt{2(n+1)}\zeta_{n+1}(x).
  \]
  See e.g. \cite[(1.1.5) p.2]{thangavelu1993lectures} for the corresponding relation for the non-normalized Hermite functions.\\
  
  STEP 3: {\sl In the $\mathcal H_{-3}$ topology it holds that
    \[
\lim_{h\rightarrow 0}\frac{B_{\lambda_1+h}-B_{\lambda_1}}{h}-N_{\lambda_1}=0.
\]
}

Using \eqref{Hen} for $k=3$ and \eqref{bound} we have for $|h|\le 1$
\[
  \begin{split}
    \left|\frac{B_{\lambda_1+h}-B_{\lambda_1}}{h}-N_{\lambda_1}\right|^2_{\mathcal H_{-3}}
    &=\sum_{n=0}^\infty\frac{1}{8(n+1)^3}
\left|\frac{1}{h}
      \int_{\lambda_1}^{\lambda_1+h}\left(\zeta_n(u)-
        \zeta_n(\lambda_1)\right)du\right|^2\\
    &\le\sum_{n=0}^\infty\frac{1}{8(n+1)^3}
\frac{1}{h^2}   \left(   \int_{\lambda_1}^{\lambda_1+h}\max_{u\in[\lambda_1,\lambda_1+h]}|\zeta_n(u)-      \zeta_n(\lambda_1)|du
\right)^2\\
    &=\sum_{n=0}^\infty\frac{1}{8(n+1)^3}\left(\max_{u\in[\lambda_1,\lambda_1+h]}|\zeta_n(u)-      \zeta_n(\lambda_1)|\right)^2\\
          &=\sum_{n=0}^\infty\frac{1}{8(n+1)^3}
\left(\max_{u\in[\lambda_1,\lambda_1+h]}|\int_u^{\lambda_1}\zeta^\prime_n(u)du|\right)^2\\
      &\le\sum_{n=0}^\infty\frac{1}{8(n+1)^3}\left(\max_{u\in[\lambda_1,\lambda_1+h]}h^2\cdot|\zeta^\prime_n(u)|\right)^2\\
          &\le h^2\sum_{n=0}^\infty\frac{1}{8(n+1)^3}(C(|\lambda_1|+1+\sqrt{2(n+1)})^2
\end{split}
\]
which goes to $0$ as $h\rightarrow 0$ since $\sum_{n=0}^\infty\frac{(\sqrt{2(n+1)})^2}{(n+1)^3}<\infty$.\smallskip

STEP 4: {\sl The limit in Step 3 is in the strong topology of the Kondratiev space $\mathcal K^\prime$.}\smallskip

Since the real line is a metric space it is sufficient to consider sequences to compute the limit in the strong topology of $\mathcal K^\prime$. The result follows then from
Proposition \ref{prop} and the previous step.
\end{proof}

\section{The fractional $\mathbb H$-Brownian motion}
\setcounter{equation}{0}
\label{sec-fbm}
Recall the definition of the classical real-valued stationary increment Gaussian processes with covariance \eqref{k-t-s-1}, with $d\sigma(u)=m(u)du$ and
$m(u)$ even,
\[
K_{r}(t,s)=\int_{\mathbb{R}}\frac{(1-\cos tu)(1-\cos su)+(\sin tu)(\sin su)}{u^{2}}m(u)du,\quad t,s\in\mathbb R.
\]
We now consider the counterpart of these processes in the hyperbolic setting. We use the setting of Corollary \ref{cor-78} and
let $Z_{1},Z_{2},\ldots$ and $W_1,W_2,\ldots$ be constructed as there from two {\sl a priori} different cylindrical probabilities. We set
\begin{eqnarray}
\label{9-6}
  B_{t}^{(1)}(\w) & = & \sum_{n=0}^{\infty}\left(\int_{\mathbb{R}}\frac{1-\cos tu}{u}\zeta_{n}(u)\sqrt{m(u)}du\right)Z_{2n}(\w)\\
B_{t}^{(2)}(\w) & = & \sum_{n=0}^{\infty}\left(\int_{\mathbb{R}}\frac{1-\cos tu}{u}\zeta_{n}(u)\sqrt{m(u)}du\right)W_{2n}(\w)\\
  B_{t}^{(3)}(\w) & = & \sum_{n=0}^{\infty}\left(\int_{\mathbb{R}}\frac{\sin tu}{u}\zeta_{n}(u)\sqrt{m(u)}du\right)
                        Z_{2n+1}(\w)\\
  B_{t}^{(4)}(\w) & = & \sum_{n=0}^{\infty}\left(\int_{\mathbb{R}}\frac{\sin tu}{u}\zeta_{n}(u)\sqrt{m(u)}du\right)
                        W_{2n+1}(\w)
                        \label{9-9}
\end{eqnarray}

\begin{theorem} In the above notation, the $\mathbb{H}$-valued
process 
\begin{equation}
B_{p}^{m}(\w)=U\begin{pmatrix}B_{\lambda_{1}}^{(1)}(\w)+B_{\lambda_{1}}^{(3)}(\w) & 0\\
0 & B_{\lambda_{2}}^{(2)}(\w)+B_{\lambda_{2}}^{(4)}(\w) 
\end{pmatrix}U,\quad p,q\in\mathbb H,
\end{equation}
is centered Gaussian with covariance function 
\begin{equation}
K_{m}(p,q)=\int_{\mathbb{R}}\frac{(1-\cos pu)(1-\cos qu)+(\sin pu)(\sin qu)}{u^{2}}m(u)du
\end{equation}
\end{theorem}
\begin{proof}
We proceed in a number of steps.\\

STEP 1: \textsl{The functions $\varphi_{t}$ and $\psi_{t}$ defined
a.e. on $\mathbb{R}$ by 
\begin{eqnarray}
\varphi_{t}(u) & = & \frac{1-\cos tu}{u}\zeta_{n}(u)\sqrt{m(u)}\\
\psi_{t}(u) & = & \frac{\sin tu}{u}\zeta_{n}(u)\sqrt{m(u)}
\end{eqnarray}
where $\zeta_{1},\zeta_{2},\ldots$ denote the normalized Hermite functions (see \eqref{hnnorm}),
belong to ${L}^{2}(\mathbb{R})$.}\smallskip{}

 This follows directly from the bound \eqref{m-u-2}.\\

STEP 2: \textsl{The proof is concluded by remarking that $B^{(1)}$ and $B^{(3)}$ are independent with respect to $P_\mu$ , and  $B^{(2)}$ and $B^{(4)}$.
  are independent with respect to $P_\sigma$}

\end{proof}

When $m(u)=\frac{|u|^{1-2H}}{2\pi}$ the preceding theorem becomes

\begin{theorem} The $\mathbb{H}$-valued process 
\begin{equation}
B_{p}^{H}(\w)=U\begin{pmatrix}B_{\lambda_{1}}^{(1)}(\w)+B_{\lambda_{1}}^{(3)}(\w) & 0\\
0 & B_{\lambda_{1}}^{(2)}(\w)+B_{\lambda_{1}}^{(4)}(\w) 
\end{pmatrix}U
\end{equation}
is centered Gaussian with covariance function \eqref{K-H-t-s}: 
\[
K_{H}(p,q)=\frac{V_H}{\pi}\left\{ |p|_{\mathbb{H}}^{2H}+|q|_{\mathbb{H}}^{2H}-|p-q|_{\mathbb{H}}^{2H}\right\}
\]
where $V_{H}$ is given by \eqref{V-H}.
\end{theorem}

Here too one can compute the derivative of the process in the topological vector space $\mathcal K^\prime$, under appropriate hypothesis on $m$; see  \cite{aal2,aal3}
for the classical setting.

\begin{remark}
  Figure \ref{fig-1} was drawn by one of the co-authors (LMA) and already appears in \cite{MR4998574}.
  \end{remark}
 \bibliographystyle{plain}
 \def\cprime{$'$} \def\cprime{$'$} \def\cprime{$'$}
  \def\lfhook#1{\setbox0=\hbox{#1}{\ooalign{\hidewidth
  \lower1.5ex\hbox{'}\hidewidth\crcr\unhbox0}}} \def\cprime{$'$}
  \def\cprime{$'$} \def\cprime{$'$} \def\cprime{$'$} \def\cprime{$'$}
  \def\cprime{$'$}

\end{document}